\documentclass[a4paper, 10pt, reqno, DIV=calc]{amsart}

\usepackage[utf8]{inputenc}
\usepackage[T1]{fontenc}
\usepackage{lmodern}
\usepackage{dsfont}
\usepackage{microtype}
\usepackage{amsmath, amsthm, thmtools}
\usepackage{mathrsfs}								
\usepackage{esint}									
\usepackage{enumitem}								
\usepackage{listings}								
\usepackage{stmaryrd}								
\usepackage[mathscr]{eucal}
\newcommand\mathcall[1]{\text{\usefont{U}{BOONDOX-cal}{m}{n}#1}}

\usepackage[dvipsnames]{xcolor}
\usepackage[linktocpage]{hyperref}
\definecolor{colorred}{HTML}{B00000}
\definecolor{colorgreen}{HTML}{258300}
\hypersetup{colorlinks=true, linkcolor=colorred, citecolor=colorgreen, urlcolor=black}								
\usepackage{xifthen}								
\usepackage{url}
\usepackage{booktabs}								
\usepackage{todonotes}
\usepackage{caption}
\usepackage{bbm}									
\usepackage{stmaryrd}								
\makeatletter
\newcommand\MyAutoefPhrasecolorGroup[1]{%
  \color@begingroup\color{MyCurrentcolor}#1\endgroup
}%
\def\HyRef@testreftype#1.#2\\{%
 \colorlet{MyCurrentcolor}{.}%
 \ltx@IfUndefined{#1autorefname}{%
   \ltx@IfUndefined{#1name}{%
     \HyRef@StripStar#1\\*\\\@nil{#1}%
     \ltx@IfUndefined{\HyRef@name autorefname}{%
       \ltx@IfUndefined{\HyRef@name name}{%
         \def\HyRef@currentHtag{}%
         \Hy@Warning{No autoref name for `#1'}%
       }{%
         \edef\HyRef@currentHtag{%
           \noexpand\MyAutoefPhrasecolorGroup{%
             \expandafter\noexpand\csname\HyRef@name name\endcsname
           }%
           \noexpand~%
         }%
       }%
     }{%
       \edef\HyRef@currentHtag{%
         \noexpand\MyAutoefPhrasecolorGroup{%
           \expandafter\noexpand
           \csname\HyRef@name autorefname\endcsname
         }%
         \noexpand~%
       }%
     }%
   }{%
     \edef\HyRef@currentHtag{%
       \noexpand\MyAutoefPhrasecolorGroup{%
         \expandafter\noexpand\csname#1name\endcsname
       }%
       \noexpand~%
     }%
   }%
 }{%
   \edef\HyRef@currentHtag{%
     \noexpand\MyAutoefPhrasecolorGroup{%
       \expandafter\noexpand\csname#1autorefname\endcsname
     }%
     \noexpand~%
   }%
 }%
}%
\makeatother

\numberwithin{equation}{section}

\newcommand{\nlb}{{\ensuremath{\textnormal{b}}}}
\newcommand{\nlc}{{\ensuremath{\textnormal{c}}}}
\newcommand{\nld}{{\ensuremath{\textnormal{d}}}}

\newcommand{\nlC}{{\ensuremath{\textnormal{C}}}}

\newcommand{\rmd}{{\ensuremath{\mathrm{d}}}}
\newcommand{\rme}{{\ensuremath{\mathrm{e}}}}

\newcommand{\rmr}{{\ensuremath{\mathrm{r}}}}

\newcommand{\rmE}{{\ensuremath{\mathrm{E}}}}

\newcommand{\rmS}{{\ensuremath{\mathrm{S}}}}

\newcommand{\sfa}{{\ensuremath{\mathsf{a}}}}
\newcommand{\sfb}{{\ensuremath{\mathsf{b}}}}

\newcommand{\sfd}{{\ensuremath{\mathsf{d}}}}

\newcommand{\sfp}{{\ensuremath{\mathsf{p}}}}

\newcommand{\sfB}{{\ensuremath{\mathsf{B}}}}

\newcommand{\sfD}{{\ensuremath{\mathsf{D}}}}

\newcommand{\sfL}{{\ensuremath{\mathsf{L}}}}

\newcommand{\sfP}{{\ensuremath{\mathsf{P}}}}

\newcommand{\cald}{{\ensuremath{\mathcall{d}}}}

\newcommand{\calk}{{\ensuremath{\mathcall{k}}}}
\newcommand{\call}{{\ensuremath{\mathcall{l}}}}

\newcommand{\caln}{{\ensuremath{\mathcall{n}}}}

\newcommand{\calD}{{\ensuremath{\mathcall{D}}}}
\newcommand{\calE}{{\ensuremath{\mathcall{E}}}}
\newcommand{\calF}{{\ensuremath{\mathcall{F}}}}

\newcommand{\calL}{{\ensuremath{\mathcall{L}}}}
\newcommand{\calM}{{\ensuremath{\mathcall{M}}}}

\newcommand{\calU}{{\ensuremath{\mathcall{U}}}}
\newcommand{\calV}{{\ensuremath{\mathcall{V}}}}

\newcommand{\N}{\boldsymbol{\mathrm{N}}}						
\newcommand{\Z}{\boldsymbol{\mathrm{Z}}}						
\newcommand{\R}{\boldsymbol{\mathrm{R}}}						
\renewcommand{\d}{\,\mathrm{d}}				

\DeclareMathOperator*{\esssup}{esssup}		
\DeclareMathOperator{\supp}{spt}			

\let\originalleft\left			
\let\originalright\right
\renewcommand{\left}{\mathopen{}\mathclose\bgroup\originalleft}
\renewcommand{\right}{\aftergroup\egroup\originalright}

\newcommand{\mapdef}[3][]{\ifthenelse{\isempty{#1}}{#2\quad\longmapsto\quad #3}{#1\colon\quad #2\quad\longmapsto\quad #3}}		
\newcommand{\der}[2][]{\ifthenelse{\isempty{#1}}{\frac{\nld}{\nld #2}}{\left.\frac{\nld}{\nld #2}\right\vert_{#1}}}				
\makeatletter
\newcommand{\checknarg}{\@ifnextchar\bgroup{\gobblenarg}{}}
\newcommand{\gobblenarg}[1]{\@ifnextchar\bgroup{,\ \! #1\gobblenarg}{,\ \! #1}}

\theoremstyle{definition}
\newtheorem{bump}{Bump}[section]

\theoremstyle{plain}
\newtheorem{theorem}[bump]{Theorem}
\newtheorem{proposition}[bump]{Proposition}
\newtheorem{definition}[bump]{Definition}
\newtheorem{lemma}[bump]{Lemma}
\newtheorem{corollary}[bump]{Corollary}

\theoremstyle{remark}
\newtheorem{remark}[bump]{Remark}
\newtheorem{example}[bump]{Example}

\newtheoremstyle{cited}
{\topsep}		
{\topsep}		
{\itshape}		
{}				
{\bfseries}		
{\textbf{.}}	
{.5em}			
{\thmname{#1} \thmnumber{#2} \thmnote{\normalfont#3}}		

\theoremstyle{cited}			

\makeatletter
\let\@fnsymbol\@arabic	 		
\makeatother

\makeatletter
\def\nonumberfootnote{\xdef\@thefnmark{}\@footnotetext}			
\makeatother

\makeatletter
\@namedef{subjclassname@2020}{\textup{2020} Mathematics Subject Classification}
\makeatother

\newcommand{\mms}{\mathit{M}}				
\newcommand{\mmms}{\mathit{N}}				
\newcommand{\met}{\sfd}						
\newcommand{\Rmet}{g}					
\newcommand{\RRmet}{h}					
\newcommand{\meas}{\mathfrak{m}}			
\newcommand{\mmeas}{\mathfrak{n}}		
\newcommand{\Leb}{\calL}					
\newcommand{\vol}{\mathfrak{v}}				
\newcommand{\Prob}{\boldsymbol{\mathrm{P}}}					
\newcommand{\Exp}{\boldsymbol{\mathrm{E}}}	
\newcommand{\F}{\calF}						
\newcommand{\Kato}{\mathscr{K}}

\newcommand{\RCD}{\mathrm{RCD}}				
\newcommand{\BE}{\mathrm{BE}}				

\newcommand{\bounded}{\nlb}					
\newcommand{\comp}{\nlc}					
\newcommand{\loc}{\mathrm{loc}}				
\newcommand{\HS}{{\mathrm{HS}}}				
\newcommand{\Ric}{\mathrm{Ric}}				
\newcommand{\Cont}{\nlC}					
\newcommand{\Ell}{\mathit{L}}				
\newcommand{\Lip}{\mathrm{Lip}}				

\newcommand{\Ch}{\calE}						
\newcommand{\Dom}{\mathcall{D}}					
\newcommand{\Ball}{\sfB}

\DeclareMathOperator{\Hess}{Hess}			
\newcommand{\ChHeat}{\sfP}					
\newcommand{\Schr}[1]{\sfP^{#1}}			
\newcommand{\push}{\sharp}					

\newcommand{\One}{1}				

\usepackage{blindtext}

\providecommand{\bysame}{\leavevmode\hbox to3em{\hrulefill}\thinspace}

\let\oldtocsection=\tocsection

\let\oldtocsubsection=\tocsubsection

\let\oldtocsubsubsection=\tocsubsubsection

\renewcommand{\tocsection}[2]{\hspace{0em}\oldtocsection{#1}{#2}}
\renewcommand{\tocsubsection}[2]{\hspace{1em}\oldtocsubsection{#1}{#2}}
\renewcommand{\tocsubsubsection}[2]{\hspace{2em}\oldtocsubsubsection{#1}{#2}}

\newcommand{\nocontentsline}[3]{}
\newcommand{\tocless}[2]{\bgroup\let\addcontentsline=\nocontentsline#1{#2}\egroup}

\usepackage{todonotes}

\begin{document}
\title[Heat kernel bounds and Ricci curvature]{Heat kernel bounds and Ricci curvature for Lipschitz manifolds}
\author{Mathias Braun}
\address{University of Bonn, Institute for Applied Mathematics, Endenicher Allee 60, 53115 Bonn, Germany}
\email{braun@iam.uni-bonn.de}
\author{Chiara Rigoni}
\address{Faculty of Mathematics, University of Vienna, 
Oskar-Morgenstern-Platz 1, 1090 Wien, Austria}
\email{chiara.rigoni@univie.ac.at}
\date{\today}
\thanks{}

\begin{abstract} Given any $d$-dimensional Lipschitz Riemannian manifold $(M,g)$ with heat kernel $\mathsf{p}$, we establish uniform upper bounds on $\sfp$ which can always be  decoupled in space and time. More precisely, we prove the existence of a constant $C>0$ and a bounded Lipschitz function $R\colon M \to (0,\infty)$ such that for every $x\in M$ and every $t>0$, 
	\begin{align*}
		\sup_{y\in M} \mathsf{p}(t,x,y) \leq C\min\{t, R^2(x)\}^{-d/2}.
	\end{align*} 
	This allows us to identify suitable weighted Lebesgue spaces w.r.t.~the given volume measure as subsets of the Kato class induced by $(M,g)$. In the case $\partial M \neq \emptyset$, we also provide an analogous inclusion for Lebesgue spaces w.r.t.~the surface measure on $\partial M$. 
	
	We use these insights  to give sufficient conditions for a possibly noncomplete Lipschitz Riemannian manifold to be tamed, i.e.~to admit a measure-valued lower bound on the Ricci curvature, formulated in a synthetic sense.
\end{abstract}
\maketitle
\thispagestyle{empty}

\tableofcontents

\section{Introduction}

\subsubsection*{Background} A \emph{Lipschitz manifold}  is a topological manifold, possibly with boundary, with locally Lipschitz transition maps, cf.~\autoref{Sub:Calculus on}  for details and basic notions. Every smooth, $\smash{\Cont^1}$, or piecewise linear manifold has a Lipschitz structure. Even better, by Sullivan's theorem \cite{sullivan1979}  \emph{every} topological manifold, apart from dimension $4$, can be endowed with a Lipschitz structure which is unique up to  locally bi-Lipschitz homeomorphisms isotopic to the identity. (Both the existence and the uniqueness statements are false in dimension $4$, as proven in \cite{donaldson1989}.) While this result follows from smoothability of all topological manifolds of dimension less than $4$ \cite{moise1977}, it is a key feature in  higher dimensions, where $\smash{\Cont}^1$ structures might not exist at all  \cite{kervaire1960}. An additional strength of Sullivan's  result comes to light in the context of \textit{Lipschitz Riemannian manifolds}, briefly LR manifolds, which are pairs $(\mms,\Rmet)$ of a Lipschitz manifold $\mms$ equipped with a compatible Riemannian metric $\Rmet$.  
As summarized in \autoref{Sub:Calculus on}, such an $(\mms,\Rmet)$ automatically induces  a length  metric space $(\mms,\met)$ \cite{DCP, DCP2, dececco1991, luukkainen1977} and a Dirichlet form $(\Ch, W^{1,2}(\mms))$ \cite{chen2012, davies1989} whose    correspondence, however, is not one-to-one in general \cite{sturm1997}. (Unless otherwise specified, all Lebesgue and Sobolev spaces are considered w.r.t.~the volume measure $\vol$ on $\mms$ induced by $\Rmet$.) In particular, these notions allow for canonical definitions of Lipschitz functions, geodesics, heat flow and Brownian motion. A further feature of LR manifolds is their well-behavedness under bi-Lipschitz homeomorphisms. Even more geometrically, every compact  Alexandrov space of finite Hausdorff dimension contains an open convex subset of full Hausdorff measure which has an LR structure \cite{perelman2017}. There is thus high evidence in investigating LR  manifolds from probabilistic, analytic, and geometric perspectives.

Remarkably, however, few attention has been devoted to this  setting so far (even in cases where the considered  LR manifolds are ``almost smooth'' in a suitable sense, compare e.g.~with \autoref{Def:Almost smooth} below, but whose regular parts then typically fail to be e.g.~complete). Relevant works study differential forms \cite{goldshtein1982, sullivan1979, whitney1957}, Hodge--de Rham's theorem and index theory \cite{hilsum1985, teleman1983}, eigenvalue estimates \cite{lott2018}, boundary-value problems \cite{goldshtein2011},  smoothability \cite{heinonen2011, whitehead1961}, or Varadhan   short-time asymptotics for the heat kernel \cite{norris1997}. 

In view of the above mentioned diversity of examples as well as recent breakthroughs in metric geometry for nonuniform  curvature bounds \cite{braun2021,carron2021, ERST20, guneysu2020, sturm2020}, we predict the class of LR manifolds --- regarded as prototypes of Riemannian spaces with singularities  --- to have great potential for near future research, which this article aims to initiate.

\subsubsection*{Upper heat kernel bounds} The heat flow $(\ChHeat_t)_{t \geq 0}$ corresponding to $\Ch$ can be represented as an integral operator by a jointly locally Hölder integral kernel $\sfp\colon (0,\infty)\times\mms^2\to (0,\infty)$ \cite{sturm1995,sturm1996}, as we recall in  \autoref{Sub:Some potential theory}. It is reluctantly accepted, even in the smooth setting  \cite{grigoryan2009}, that $\sfp$ is explicitly computable only in few cases. Nevertheless, in practice it typically suffices to know upper bounds on $\sfp$. Usual estimates of this sort are Gaussian and read as
\begin{align}\label{Eq:GUP}
\sfp(t,x,y) \leq C\,\vol\big[\Ball_{\sqrt{t}}(x)\big]^{-1/2}\,\vol\big[\Ball_{\sqrt{t}}(y)\big]^{-1/2}\,\exp\!\Big[\!-\!\frac{\met^2(x,y)}{(4+\varepsilon)t}\Big]
\end{align}
for every $\varepsilon > 0$ (see \autoref{Pr:HK bound general} and the references in \autoref{Sub:Heat kernel}). 

This is still unsatisfactory for at least three  reasons, which to address is of fundamental importance for our later applications, as detailed in the next paragraph.
\begin{itemize}
\item  First, the constant $C>0$ and the admissible range of $t>0$ do in general depend on $x,y\in\mms$, in which case \eqref{Eq:GUP} becomes qualitatively useless when, for example, one wants to derive globally or locally uniform upper bounds, both in space and time, on the heat kernel $\sfp$. 
\item Second, the dependency of the r.h.s.~of \eqref{Eq:GUP} on $t$, $x$, and $y$ can usually not be decoupled. 
\item Third, one would often like upper bounds on $\sfp(t,x,y)$ to be independent of $y$, which cannot be guaranteed by \eqref{Eq:GUP} either. 
\end{itemize}
The argument for \eqref{Eq:GUP} from \cite{sturm1995} works for general Dirichlet spaces satisfying a local doubling condition \cite[p.~293]{sturm1995} and a local Sobolev inequality \cite[p.~294]{sturm1995}. However, our given LR structure of $(\mms,\Rmet)$ is a priori respected only by the fact that these two properties transfer from Euclidean space back to $\mms$ locally in charts, usually with nonuniform constants, cf.~\autoref{rmk:lambda} and \autoref{Sub:Heat kernel}. Settings in which the respective constants are uniform and yield versions of \eqref{Eq:GUP} addressing all above mentioned issues, e.g.~as described in \autoref{Sub:Quasi isom} (following standard arguments, see e.g.~\cite{grigoryan2009,saloffcoste1992}), are too restrictive to be satisfied by all LR manifolds.  

The following main result of our work, detailed in  \autoref{Prop:heat kernel bound}, establishes a slightly different upper bound on $\sfp$ for \emph{general} LR manifolds which rules out all three problems described above. Let us set $d :=\dim\mms$.

\begin{theorem}\label{Th:IntroTh1} There exist a constant $C>0$, depending only on $d$, and a bounded Lipschitz function $R\colon \mms \to (0,\infty)$ such that for every $x,y\in\mms$ and every $t>0$,
\begin{align*}
\sfp(t,x,y) \leq C\min\{t,R^2(x)\}^{-d/2}.
\end{align*}
\end{theorem}

To the best of our knowledge, so far \autoref{Th:IntroTh1} has only been proven for smooth Riemannian manifolds without boundary \cite{guneysu2017b}. In particular, \autoref{Th:IntroTh1} seems to be new even in the setting of smooth manifolds with boundary. 

The proof of \autoref{Th:IntroTh1}, presented in \autoref{Ch:Control pairs}, and  its idea have been inspired by \cite[Thm.~15.4]{grigoryan2009} and  \cite[Thm.~IV.14]{guneysu2017}. We first prove the existence of some $a>0$ such that, given any $x\in\mms$, we have the \emph{Faber--Krahn inequality}
\begin{align}\label{Eq:FK intro}
\lambda_1(O) \geq a\,\vol[O]^{-2/d}
\end{align}
for every open, relatively compact $O\subset\mms$ contained in a suitable neighbor\-hood of $x$. (Here, $\lambda_1(O)$ denotes the first Dirichlet eigenvalue on $O$.) 
Indeed, given any fixed $b>0$, modulo a reflection argument in the case $\partial\mms \neq \emptyset$, we prove that  $(\mms,\Rmet)$ locally looks like a Euclidean space $\smash{(\R^d,\RRmet)}$ up to a rescaling of $\Rmet$, cf.~\autoref{Le:FKM}. Here $\RRmet$ is an arbitrary Riemannian metric  equivalent to the standard Euclidean metric tensor $\smash{\Rmet^\rmE}$, i.e.~$\smash{b^{-1}\,\Rmet^\rmE \leq \RRmet \leq b\,\Rmet^\rmE}$, w.r.t.~which \eqref{Eq:FK intro} is well-known to hold on $\smash{(\R^d,\RRmet)}$ \cite{grigoryan2009,saloffcoste1992}.  Of course, the size of the domain itself in which $O$ is allowed to be situated may vary over $\mms$. However, it is controllable in terms of the \emph{Euclidean radius} $\smash{r_\rmE}$, cf. \autoref{Def:Eucl rad}: given $x\in\mms$, $r_\rmE(x,b)$ is naturally defined as the largest possible radius $r>0$ w.r.t.~which, roughly speaking, $\Ball_r(x)$ is quasi-isometric, with constants $\smash{b^{-1}}$ and $b$, to a subset of $\smash{(\R^d,\RRmet)}$.  (As shown in \autoref{lemma:propr}, truncated versions $R$ of $\smash{r_\rmE}$ obey good  regularity properties.) This  allows us to pull \eqref{Eq:FK intro} locally back to $\mms$ while retaining  a certain uniformity in $a$, which is possible by the fact that $b$ is fixed. Then, to get a (Li--Yau-type) heat kernel estimate which takes this spacial variation into account, we use an $\Ell^1$-mean value inequality for local subsolutions to the heat equation (see \autoref{Th:L^p mean value}) derived from its $\Ell^2$-counterpart (see \autoref{Th:L^2 mean value}) by Li--Wang's $\Ell^2$-to-$\Ell^1$ reduction  \cite{guneysu2017b,li1999}.

We point out that these arguments crucially exploit the LR structure of $(\mms,\Rmet)$, whence this setting seems to be as general as possible to prove \autoref{Th:IntroTh1}. Furthermore, \autoref{Th:IntroTh1} is quite concrete and well-suited for applications, in the sense that $r_\rmE$ --- hence $R$ --- is  in general  easy to bound from below  explicitly.

\subsubsection*{Kato class and heat kernel control pairs} One of our main motivations to establish \autoref{Th:IntroTh1} and to address the mentioned nonuniformities in \eqref{Eq:GUP}  is to understand better the Kato class of LR manifolds $(\mms,\Rmet)$.

\begin{definition}\label{Def:KatoIntro} The \emph{Kato class} $\Kato(\mms)$ consists of all signed Borel measures $\kappa$ on $\mms$ which do not charge $\Ch$-polar sets and such that
\begin{align*}
\lim_{t\to 0} \sup_{x\in\mms} \int_0^t\!\!\int_\mms\sfp(s,x,\cdot)\d\vert\kappa\vert\d s =0.
\end{align*}
\end{definition}

The prominent role of the Kato class in mathematical physics has been well established in the past decades in a series of works \cite{aizenman1982,kato1972,stollmann1996}. Recently, its interest in the context of singular Ricci bounds has grown as well \cite{braun2021, braun2020, carron2019, carronrose, ERST20, guneysu2017, guneysu2020, rose2019} (also with applications to Ricci flow and general relativity). The study of the class of LR manifolds in this latter setting constitutes the last part of our work.

It is easy to see that $\Ell^\infty(\mms)\,\vol \subset\Kato(\mms)$. To verify from \autoref{Def:KatoIntro} that a given $\kappa$ does belong to $\Kato(\mms)$ beyond this trivial case, the importance of having appropriate upper bounds for $\sfp$ according to the criteria raised above is  intuitively clear. In concrete words, in \autoref{Th:Kato criterion} we  identify an explicit subspace of $\Kato(\mms)$ for a given \emph{heat kernel control pair} $(\Xi,\Sigma)$, i.e.~an ordered pair of suitable functions $\Xi\colon \mms\to (0,\infty)$ and $\Sigma\colon (0,1]\to (0,\infty)$ for which there exists some $C>0$ such that
\begin{align*}
\sfp(t,x,y) \leq C\,\Sigma(t)\,\Xi(x)
\end{align*}
for every $x,y\in\mms$ and every $t\in (0,1]$, see \autoref{Def:Heat kernel control pair}. 

\begin{theorem}\label{Th:In2} Suppose that $d\geq 2$, and assume the  existence of some heat kernel control pair $(\Xi,\Sigma)$. Then $\smash{\Ell^p(\mms, \Xi\,\vol)\,\vol + \Ell^\infty(\mms)\,\vol \subset \Kato(\mms)}$ for every $p\in (d/2,\infty)$. 
\end{theorem}

By an evident choice of $(\Xi,\Sigma)$, this result recovers the well-known fact that $\smash{\Ell^p(\R^d)\,\mathcall{L}^d\subset \Kato(\R^d)}$ for every $p\in (d/2,\infty]$ \cite{aizenman1982}.

As an immediate consequence of our \autoref{Th:IntroTh1}, the second hypothesis of \autoref{Th:In2} is \emph{always} satisfied,  cf.~\autoref{prop:Existence control pair}: this in turn implies that, up to the evident identification, $\Kato(\mms)$ always contains certain weighted $\Ell^p$-spaces.

\begin{corollary}
Every LR manifold admits a canonical heat kernel control pair.
\end{corollary}

 
\subsubsection*{Taming for LR manifolds} The  notion of \emph{tamed spaces}, summarized in  \autoref{Sec:Taming}, has recently been introduced in \cite{ERST20}. It offers a synthetic way of speaking about the Ricci curvature of a  Dirich\-let space being bounded from below by a distribution $\kappa$. This machinery provides a far-reaching generalization of the Bakry--\smash{Émery} theory for diffusion operators \cite{bakry} and its Eulerian counterpart for RCD spaces \cite{ambrosio}, formulated in terms of Schrödinger operators by a weak form of the $1$-Bochner  inequality
\begin{align}\label{Eq:IntroBochner}
	(\Delta - \kappa)\vert\rmd f\vert  - \Rmet^*(\rmd f,\rmd \Delta f)\,\vert\rmd f\vert^{-1} \geq 0.
\end{align}

Since elements in $\Kato(\mms)$ provide special cases of the  relevant  distribu\-tions considered in \cite{ERST20} and since Ricci bounds have powerful  probabilistic, analytic and geometric consequences, we believe our following condition for taming of LR manifolds, cf.~\autoref{Th:Tamed convex boundary} below, to be of high interest.

\begin{theorem}\label{Th:Th}	Assume that $d\geq 2$ and that $(\mms,\Rmet)$ is \emph{almost smooth}, i.e.~it contains a $\vol$-conegligible, open, smooth Riemannian manifold $\smash{(\mms^\infty,\Rmet\big\vert_{\mms^\infty})}$. Let $\calk\in \Ell^p(\mms,\Xi\,\vol) + \Ell^\infty(\mms)$, where $p\in (d/2,\infty)$ and  $(\Xi,\Sigma)$ is given a heat kernel control pair, be a lower bound of the Ricci curvature on $\mms^\infty$, i.e. 
\begin{align*}
\Ric \geq \calk\quad\text{on }\mms^\infty.
\end{align*}
Lastly, suppose that the Laplacian $\Delta$ is essentially self-adjoint on $\smash{\Cont_\comp^\infty(\mms^\infty)}$. Then the Dirichlet space $(\mms,\Ch,\vol)$ is tamed by 
	\begin{align*}
		\kappa := \calk\,\meas.
	\end{align*}
\end{theorem}

We briefly comment on the assumptions of \autoref{Th:Th} as well as on its proof, given in \autoref{Sub:An interior} below. The almost smoothness of $(\mms,\Rmet)$ is needed to set up a $\vol$-a.e.~Bochner \emph{identity} on its smooth part, which is the starting point to derive \eqref{Eq:IntroBochner}. To integrate the inherent  Schrödinger term ``$(\Delta-2\calk)\,\vert\rmd f\vert^2/2$''  by parts and to simultaneously apply the chain rule to pass over to ``$(\Delta-\calk)\,\vert \rmd f\vert$'' we will \emph{a priori} require  the $W^{1,2}$-regularity of $\vert\rmd f\vert$ for every $f\in\Dom(\Delta)$, as discussed in \autoref{Le:Reg lemma}. The essential self-adjointness of $\Delta$ on $\smash{\Cont_\comp^\infty(\mms^\infty)}$, which ensures the previous regularity as shown in \cite{hess1980}, is then used to extend this statement to arbitrary elements of $\Dom(\Delta)$, through an integrated form of \eqref{Eq:IntroBochner}.

The hypothesis on essential self-adjointness  in \autoref{Th:Th} is less restrictive than it might seem. Indeed, \autoref{Th:Th} already applies to \emph{regular almost Riemannian structures} \cite{prandi2018}, i.e.~manifolds which admit a smooth structure, but are endowed with a possibly singular metric. Actually already in the smooth case, interesting tamed  situations are covered: indeed, essential self-adjointness holds for example in the case in which $\mms$ arises from a smooth Riemannian manifold from which a closed set of codimension at least $4$ is removed \cite{masamune1999} (see \autoref{Ex:Sing I} and \autoref{Ex:Sing II} for details). In particular, we cover  smooth examples of tamed spaces with constant lower Ricci bounds \emph{beyond}  complete manifolds, while classical  examples in Bakry--\smash{Émery's} theory typically require completeness (which is also always part of the assumptions in the correspondent framework  of RCD spaces), a property implying  essential self-adjointness \cite{strichartz}. 

Lastly, \autoref{Sub:A boundary} deals with taming of certain LR manifolds with boundary, see \autoref{Th:Taming bdry qi}. Here, compared to the $\vol$-absolutely continuous case from \autoref{Th:In2}, it is even more sophisticated to figure out when certain $\Ell^q$-spaces w.r.t.~the surface measure $\sigma$ on $\partial\mms$ belong to $\Kato(\mms)$. This is addressed in \autoref{Th:Surface measure}.

\subsubsection*{Organization} In \autoref{Sub:Calculus on}, we give a concise overview of LR manifolds and all induced analytic objects needed for our purposes. \autoref{Ch:Control pairs} introduces the concept of control pairs for the given heat kernel and shows the general existence of such a pair. Lastly, in \autoref{Ch:Taming} we briefly recapitulate basic notions about tamed spaces from \cite{ERST20} and provide sufficient conditions for taming of LR manifolds.

\subsubsection*{Acknowledgments} Both authors gratefully acknowledge funding by the European Research Council through the ERC-AdG ``Ricci\-Bounds''. The second author was partially supported by the Austrian Science Fund (FWF) through project  F65.

\section{Calculus on Lipschitz manifolds}\label{Sub:Calculus on}

This preliminary chapter is a survey over basic facts, considered as standard and thus mostly stated without proof, about the framework of LR manifolds we work in all over the paper. 

In this article, we do mostly not keep track of involved uniform constants (albeit we always try to be clear about their dependencies) and, by convention, allow them to change from line to line without reflecting this change in the notation.

\subsection{LR manifolds}\label{Sub:Lip R  mflds} Following \cite{DCP,DCP2,dececco1991,luukkainen1977, rosenberg1987,sullivan1979,teleman1983}, we first introduce  relevant notions of \emph{LR manifolds}. 

\subsubsection{Lipschitz manifolds}

\begin{definition}\label{def:Lipman}
A topological Hausdorff space $\mms$ is called \emph{\textnormal{(}$d$-dimensional\textnormal{)} Lipschitz manifold}, $d\in\N$, if its topology is second countable and if it is endowed with an atlas $\calU := (U_\alpha,\Phi_\alpha)_{\alpha \in A}$ such that 
\begin{enumerate}[label=\textnormal{\alph*.}]
\item $(U_\alpha)_{\alpha\in A}$ is an open cover of $\mms$,
\item\label{La:homeo} $\Phi_\alpha\colon U_\alpha \to \Phi_\alpha(U_\alpha)$ is a homeomorphism between $U_\alpha$ and an open subset $\Phi_\alpha(U_\alpha)$ of $\smash{\R^d}$ or $\smash{\R^{d-1}\times [0,\infty)}$ for every $\alpha\in A$, and 
\item\label{La:Lip homeo} the transition map $\smash{\Phi_{\alpha,\alpha'}\colon \Phi_\alpha(U_\alpha\cap U_{\alpha'})\to \Phi_{\alpha'}(U_\alpha\cap U_{\alpha'})}$, where
\begin{align*}
\Phi_{\alpha,\alpha'} := \Phi_{\alpha'}\circ\Phi_\alpha^{-1},
\end{align*}
is a locally Lipschitz map for every $\alpha,\alpha'\in A$.
\end{enumerate}
\end{definition}

Item \ref{La:homeo}~allows $\mms$ to possess vertices, edges or non-isolated conical points. In par\-ticular, Lipschitz manifolds are prototypes of topological manifolds and the presence of a boundary or  some corners is \emph{not} excluded. When $\mms$ is compact, the transition functions in  \ref{La:Lip homeo}~can be taken to be globally bi-Lipschitz. 

Any Lipschitz manifold is henceforth assumed to be connected and paracompact. As the topology of $\mms$ is second countable, $\calU$ may -- and thus will always -- be chosen to be locally finite with  countable sub-atlas. On the other hand, following \cite{norris1997}, completeness of $\mms$ is \emph{not} assumed 
 (cf.~\autoref{Sub:Metric structure} below).

Every Lipschitz manifold  of dimension $d$ is Lipschitz homeomorphic to a subset of $\smash{\R^{d(d+1)}}$ that admits locally bi-Lipschitz parametrizations, cf.~\autoref{Sub:Lip maps} below, by $\smash{\R^d}$ \cite[Thm.~4.2, Thm.~4.5]{luukkainen1977}.

\subsubsection{Differential forms}\label{Sub:Diff forms}  Let $\mms$ denote a Lipschitz manifold with an atlas $\calU := (U_\alpha,\Phi_\alpha)_{\alpha\in A}$.  A \emph{$k$-form}, $k\in\N$, on $\mms$ is a family $\omega := (\omega_\alpha)_{\alpha\in A}$ consisting of maps $\omega_\alpha\colon\Phi_\alpha(U_\alpha) \to \smash{\bigsqcup_{z\in \Phi_\alpha(U_\alpha)}\Lambda^k(\R^d)^*}$ with Borel coefficients such that 
\begin{align*}
\Phi_{\alpha,\alpha'}^*\,\omega_{\alpha'} = \omega_\alpha\quad\Leb^d\text{-a.e.}\quad\text{on }\Phi_\alpha(U_\alpha\cap U_{\alpha'})
\end{align*}
for every $\alpha,\alpha'\in A$. The pull-back $\smash{\Phi_{\alpha,\alpha'}^*}$ by $\Phi_{\alpha,\alpha'}$ makes sense $\Leb^d$-a.e.~by Rade\-macher's theorem \cite[p.~272]{whitney1957}. 

\begin{remark}
To be very precise, a $k$-form should be regarded as equivalence class w.r.t.~$\Leb^d$-a.e.~equality of its components $\omega_\alpha$, $\alpha\in A$, on every chart. Without further notice, in the sequel this will be the  precise interpretation of \emph{all} objects which are defined up to $\smash{\Leb^d}$-a.e.~equality on charts.
\end{remark}

Such an $\omega$  is \emph{smooth} if for every $\alpha\in A$, $\omega_\alpha$ has smooth coefficients in its local representation w.r.t.~Euclidean coordinates on $\Phi_\alpha(U_\alpha)$; if $\smash{\Phi_\alpha(U_\alpha)}$ is not  open in $\smash{\R^d}$, this means the existence of smooth extensions of all coefficients of $\omega_\alpha$ to an open subset of $\smash{\R^d}$. It is  \emph{essentially bounded} if $\smash{\sup_{\alpha\in A} \Leb^d\text{-}\!\esssup \vert \omega_\alpha\vert_{\Rmet^\rmE}\circ\Phi_\alpha(U_\alpha) < \infty}$. Here $\smash{\big\vert\cdot\big\vert_{\Rmet^\rmE}^2 := (\Rmet^\rmE)^*(\cdot,\cdot)}$, where $(\Rmet^\rmE)^*$ is the co-metric induced by the standard Euclidean metric tensor $\Rmet^\rmE$.

Multiplication of a $k$-form $\omega$ on $\mms$ by a Borel  function $\psi\colon\mms\to\R$ is understood chartwise, i.e.~$\psi\,\omega := (\psi\circ\Phi_\alpha^{-1}\,\omega_\alpha)_{\alpha\in A}$ is still a $k$-form on $\mms$. In turn, $k$-forms on $\mms$ are evidently defined to be  \emph{compactly supported} or to \emph{coincide a.e.} (on Borel subsets of $\mms$), respectively.

\subsubsection{Riemannian metrics}\label{Sub:Lip Riem mflds}

\begin{definition}\label{Def:Riem metrics} A \emph{Riemannian metric} $\Rmet := (\Rmet_\alpha)_{\alpha\in A}$ on $\mms$ consists of a family of scalar products $\Rmet_\alpha\colon \Phi_\alpha(U_\alpha) \to \smash{\bigsqcup_{z\in\Phi_\alpha(U_\alpha)}{(\R^d)^{*\otimes 2}}}$ with Borel  coefficients which is \emph{compatible} with $\calU$, i.e.~for every $\alpha,\alpha'\in A$,
\begin{align}\label{Eq:Compatibility condition}
(\Phi_{\alpha,\alpha'})_*\,\Rmet_\alpha = \Rmet_{\alpha'}\quad\Leb^d\text{-a.e.}\quad\text{on }\Phi_{\alpha'}(U_\alpha \cap U_{\alpha'}).
\end{align}
\end{definition}

Here $\smash{(\Phi_{\alpha,\alpha'})_*}$ designates the usual push-forward operation by $\smash{\Phi_{\alpha,\alpha'}}$ which, as in \autoref{Sub:Diff forms}, is well-defined $\smash{\Leb^d}$-a.e. 

Let $\Rmet := (\Rmet_\alpha)_{\alpha\in A}$ be a Riemannian metric on $\mms$. Any $k$-form $\omega$, $k\in\N$, on $\mms$ comes with a Borel function $\vert \omega \vert\colon \mms\to\R$ such that for every $\alpha\in A$,
\begin{align}\label{Eq:Ptw norm def}
\vert\omega\vert\circ\Phi_\alpha^{-1} = \vert\omega_\alpha\vert_{\Rmet_\alpha}\quad\Leb^d\text{-a.e.}\quad\text{on }\Phi_\alpha(U_\alpha).
\end{align}
Here $\smash{\big\vert\cdot\big\vert_{\Rmet_\alpha}^2 := \Rmet_\alpha^*(\cdot,\cdot)}$, where $\smash{\Rmet^*_\alpha}$ is the co-metric associated with $\Rmet_\alpha$.

If a smooth $k$-form $\omega$ on $\mms$ has compact support in $U_\alpha$, $\alpha\in A$, we set
\begin{align*}
\big\Vert \omega\big\Vert_{\Rmet^\rmE}^2 &:= \int_{\Phi_\alpha(U_\alpha)} \omega_\alpha\wedge\star\omega_\alpha = \int_{\Phi_\alpha(U_\alpha)} \big\vert\omega_\alpha\big\vert_{\Rmet^\rmE}^2 \d\Leb^d,\\
\big\Vert\omega\big\Vert_{\Rmet_\alpha}^2 &:= \int_{\Phi_\alpha(U_\alpha)} \omega_\alpha\wedge\star_\alpha\omega_\alpha = \int_{\Phi_\alpha(U_\alpha)} \big\vert \omega_\alpha\big\vert_{\Rmet_\alpha}^2\,\sqrt{\det\Rmet_\alpha}\d\Leb^d.
\end{align*}
Here, $\star$ and ${\star}_{\alpha}$ are the Hodge operators of the Euclidean metric and of $\Rmet_\alpha$,  respec\-tive\-ly. Following \cite[Sec.~1.4]{dececco1991}, we always assume the following \emph{local Lipschitz} condition on $\Rmet$: for every $\alpha\in A$, there exists a constant $c_\alpha \geq 1$ such that for every smooth $1$-form $\omega$ on $\mms$ with compact support in $U_\alpha$,
\begin{align*}
c_\alpha^{-1}\, \| \omega \|_{\Rmet^\rmE} \le \| \omega \|_{\Rmet_\alpha} \le c_\alpha\, \| \omega \|_{\Rmet^\rmE}.
\end{align*}

\begin{remark}\label{rmk:lambda} Under the previous hypothesis,  by \cite[Thm.~3.1]{DCP}, for every $\alpha\in A$ there exists a  constant $\lambda_\alpha\geq  1$ such that for every $\smash{\xi\in\R^d}$,
\begin{align*}
\lambda_\alpha^{-1}\,\vert\xi\vert^2 \leq \Rmet_\alpha(\xi,\xi) \leq \lambda_\alpha\,\vert\xi\vert^2\quad\Leb^d\text{-a.e.}\quad\text{on }\Phi_\alpha(U_\alpha).
\end{align*}
\end{remark}

\begin{definition}\label{Def:Lip R mfld}
A pair $(\mms,\Rmet)$ is called \emph{Lipschitz Riemannian manifold}, briefly LR manifold, if $\mms$ is a Lipschitz manifold and $\Rmet$ is a locally Lipschitz Riemannian metric, briefly LR metric, on it.
\end{definition}

\subsubsection{Volume measure}\label{Sub:Volume} Given any bounded Borel function $f\colon \mms\to [0,\infty)$ with support in $U_\alpha$, $\alpha\in A$, we set
\begin{align}\label{Eq:Volg}
\int_\mms f\d\vol_\Rmet := \int_{\Phi_\alpha(U_\alpha)} f\circ\Phi_\alpha^{-1}\,\sqrt{\det\Rmet_\alpha}\d\Leb^d.
\end{align}
By \eqref{Eq:Compatibility condition} and a change of variables, if $f$ is supported in $U_{\alpha'}$ as well, $\alpha'\in A$, then the r.h.s.'s of \eqref{Eq:Volg} for $\alpha$ and $\alpha'$ agree. Using a partition of unity, we extend \eqref{Eq:Volg} to all bounded, nonnegative Borelian $f$. This procedure yields a unique fully supported, $\sigma$-finite  Radon measure $\vol_\Rmet$ on $\mms$ which we will call \emph{volume measure}. Whenever the metric $\Rmet$ is understood, we simply write $\vol$ instead of $\vol_\Rmet$.

At various occasions, instead of $\vol$ one could also consider a more general Radon measure $\meas$ on $\mms$ which is \emph{locally equivalent} to $\vol$, i.e.~$\meas$ is of the form $\rme^{-2\phi}\,\vol$ for some locally bounded Borel function $\phi\colon \mms\to\R$. It will always be mentioned explicitly if a result applies to such an $\meas$ (or certain quantities defined in terms of $\meas$). 

Given any $p\in [1,\infty]$, we particularly have distinguished (local) Lebesgue spaces $\Ell^p(\mms,\meas)$ and $\smash{\Ell^p_\loc(\mms,\meas)}$ w.r.t.~$\meas$. If $\meas = \vol$, these will be  abbreviated by $\Ell^p(\mms)$ and $\smash{\Ell^p_\loc(\mms)}$, respectively.

\subsection{Metric structure}\label{Sub:Metric structure}  $(\mms,\Rmet)$ comes with a natural length distance function $\smash{\met_\Rmet}$, abbreviated by $\met$ when the dependence on $\Rmet$ is clear, which turns $\smash{(\mms,\met_\Rmet)}$ into a metric space. Here we review the cornerstones of its construction from \cite{DCP,DCP2,dececco1991,norris1997}, to which we refer for proofs and technical details. See also \autoref{Sub:Heat kernel} below.

A map $\gamma\colon [0,1]\to\mms$ is \emph{Lipschitz} if for every $\alpha\in A$ and every $a,b\in [0,1]$ with $\gamma([a,b])\subset U_\alpha$, $\gamma_\alpha \colon [a,b]\to \Phi_\alpha(U_\alpha)$ is Lipschitz, where
\begin{align*}
\gamma_\alpha :=  \Phi_\alpha\circ\gamma.
\end{align*}
As usual, the length of  $\smash{\gamma_\alpha}$ is defined by
\begin{align}\label{Eq:Lngth}
L(\gamma_\alpha) := \int_a^b \sqrt{\Rmet_\alpha\big(\dot{\gamma}_\alpha(t), \dot{\gamma}_\alpha(t)\big)\circ\gamma_\alpha(t)}\d t. 
\end{align}
However, if $\smash{\gamma([a,b])\subset U_{\alpha'}}$ for some  $\alpha'\in A\setminus\{\alpha\}$, the quantities $L(\gamma_\alpha)$ and $L(\gamma_{\alpha'})$ do  not coincide in general. For instance, albeit the set of non-diffe\-rentiabi\-lity points of the transition function  $\Phi_{\alpha,\alpha'}$ is $\smash{\Leb^d}$-negligible, the set of times at which $\gamma$ passes through the latter might  be charged by $\Leb^1$ \cite[Sec.~2.5]{DCP}. 

To overcome this issue, given $x,y\in\mms$ and  $B\subset\mms$ with $\vol[B] = 0$, we set
\begin{align*}
\Lip_B(x,y;\mms) &:= \big\lbrace \gamma\colon [0,1] \to \mms : \gamma\text{ Lipschitz},\\ 
&\qquad\qquad \gamma(0)=x,\ \gamma(1)=y,\
 \gamma_\push\Leb^1[B]=0\big\rbrace,
\end{align*}
i.e.~the class of Lipschitz curves in $\mms$ from $x$ to $y$ which are \emph{transversal} to $B$. Then $\Lip_B(x,y;\mms) \neq \emptyset$ \cite[Lem.~2.4]{DCP}, and there exists a Borel set $B^*\subset\mms$ with $\vol[B^*]=0$ such that the \emph{length} $L(\gamma)$, constructed from \eqref{Eq:Lngth} in the evident way \cite[pp.~165--166]{DCP}, is well-defined for every $\gamma\in\Lip_{B^*}(x,y;\mms)$ \cite[p.~166]{DCP}. Therefore, for $x$, $y$ and $B$ as above the quantities
\begin{align*}
\met_B(x,y) &:= \inf\!\big\lbrace L(\gamma) : \gamma\in \Lip_{B\cup B^*}(x,y;\mms)\big\rbrace,\\
\met(x,y) &:= \sup\!\big\lbrace \met_B(x,y) : B\subset\mms,\ \vol[B]=0\big\rbrace
\end{align*}
are meaningful. In fact, $\met$ is a distance on $\mms$, independent of $\calU$, which induces the initial topology on $\mms$ \cite[Thm.~4.4, Thm.~4.5, Cor.~4.6]{DCP}. Moreover, it is a length distance \cite[Thm.~3.10]{DCP2}, and hence geodesic if $\mms$ is complete.  We call $\met$ the \emph{path distance} induced by $\Rmet$. 

If $(\mms,\Rmet)$ is a smooth Riemannian manifold, then $\met$ coincides with the length distance induced by the metric tensor $\Rmet$ \cite[Thm.~5.1]{DCP}.

We denote by $\smash{\Ball_r^\Rmet(x)}$, or simply $\Ball_r(x)$ if $\Rmet$ is understood, the open ball with radius $r>0$ w.r.t.~$\met$ around $x\in\mms$.

\begin{remark}\label{Re:Equivalent metric} Every Lipschitz atlas $\calU := (U_\alpha, \Phi_\alpha)_{\alpha\in A}$ defines a local distance on the chart $U_\alpha$, $\alpha\in A$, namely the Euclidean distance induced on the chart
\[
\sfD_\alpha(x, y) = \vert  \Phi_\alpha(x) - \Phi_\alpha(y) \vert,
\]
where $x,y\in U_\alpha$. Starting from $\sfD_\alpha$, it is possible to construct a global (well-defined, since $\mms$ is connected) distance $\sfD$ on $\mms$, as shown in \cite{dececco1991, luukkainen1977}. This distance depends on the fixed atlas $\calU$, but turns out to be locally equivalent to   $\met$ \cite[Thm.~6.2]{DCP}.
\end{remark}

\subsection{Lipschitz continuity} 

\subsubsection{Lipschitz functions}\label{Sub:Lip functions} 

A function $f\colon \mms \to \R$ is  \emph{locally  Lipschitz} if every point in $\mms$ has an open neighborhood $U\subset\mms$ such that $\smash{f\big\vert_{U}}$ is Lipschitz as a  map between the metric spaces $\smash{(U,\met\big\vert_{U^2})}$ and $\smash{(\R,\vert\cdot - \cdot\vert)}$ for every $\alpha\in A$. By \autoref{Re:Equivalent metric}, this is equivalent to the possibly more common definition \cite{dececco1991, norris1997} by requiring local Lipschitz continuity of $f_\alpha\colon \Phi_\alpha(U_\alpha) \to \R$ for every $\alpha\in A$, where
\begin{align*}
f_\alpha := f\circ\Phi_\alpha^{-1}.
\end{align*}
Furthermore, such an $f$ is \emph{Lipschitz} if it is Lipschitz as a map between the metric spaces $(\mms,\met)$ and $(\R,\vert\cdot - \cdot\vert)$. (In particular, unlike e.g.~\cite{norris1997} Lipschitz continuity is always intended \emph{globally}.) Denote the spaces of all (locally) Lipschitz functions $f\colon \mms\to\R$  by $\Lip_\loc(\mms)$ and $\Lip(\mms)$, respectively. In turn, the spaces $\smash{\Lip_\loc(\mms;\R^{\cald})}$ and $\smash{\Lip(\mms;\R^{\cald})}$, $\cald\in\N$, consist of all  functions $\smash{f\colon\mms\to\R^{\cald}}$ whose components are (locally) Lipschitz. 

Clearly, by \autoref{Sub:Metric structure}, every element in $\Lip_\loc(\mms)$ is continuous, and if $\mms$ is compact, then $\Lip_\loc(\mms) = \Lip(\mms)$. 

By Rademacher's theorem and the chain rule, for every $f\in\Lip_\loc(\mms)$ the family $\rmd f := (\rmd f_\alpha)_{\alpha\in A}$ is a $1$-form according to \autoref{Sub:Diff forms}, the  \emph{differential} of $f$. If $f\in\Lip(\mms)$, then $\vert\rmd f\vert$ belongs to $\Ell^\infty(\mms)$.

Evidently, the induced linear operator $\rmd$ on $\Lip_\loc(\mms)$ obeys all expected locality, chain and Leibniz rules (as usual, in appropriate a.e.~senses).





\subsubsection{Lipschitz maps}\label{Sub:Lip maps} Analogously to \autoref{Sub:Lip Riem mflds} above, a map $F\colon \mms\to\mmms$ is evidently defined to be \emph{\textnormal{(}locally\textnormal{)} Lipschitz} in the metric sense. Such an $F$ is locally Lipschitz if and only if for every $\alpha\in A$ and every $\beta\in B$ with $F(U_\alpha)\cap V_\beta\neq \emptyset$, the map $F_{\alpha,\beta} \colon \Phi_{\alpha}(U_\alpha) \to \Psi_\beta(F(U_\alpha)\cap V_\beta)$ is locally Lipschitz as a map between subsets of $d$- and $\cald$-dimensional Euclidean spaces, where 
\begin{align*}
F_{\alpha,\beta}:= \Psi_\beta\circ F\circ \Phi_{\alpha}^{-1}.
\end{align*}
The spaces of all such $F$ are denoted by $\Lip_\loc(\mms;\mmms)$ and $\Lip(\mms;\mmms)$, respectively. For $\smash{\mmms = \R^{\cald}}$, the preceeding notions of (local) Lipschitz continuity are compatible with those of  \autoref{Sub:Lip functions}. 


Finally, a map $F\colon \mms\to \mmms$ is called  \emph{\textnormal{(}locally\textnormal{)} bi-Lipschitz} if $F\in \Lip_\loc(\mms;\mmms)$ or $F\in\Lip(\mms;\mmms)$,  and $F$ admits an inverse map belonging to $\Lip_\loc(\mmms;\mms)$ or $\Lip(\mmms;\mms)$, respectively.

\subsection{Quasi-isometry} 

In the following, $(\mms,\Rmet)$ and $(\mmms,\RRmet)$ will denote two LR manifolds of not necessarily equal dimensions $d$ and $\cald$, respectively, $d,\cald\in\N$, according to the notions introduced in the \autoref{Sub:Lip R  mflds}, cf.~\autoref{Def:Lip R mfld}.  Let $\calU := (U_\alpha,\Phi_\alpha)_{\alpha\in A}$ and $\calV := (V_\beta,\Psi_\beta)_{\beta\in B}$ denote their respective atlases.

\begin{definition}\label{Def:QI} We say that $(\mms,\Rmet)$ and $(\mmms,\RRmet)$ are \emph{quasi-isometric}, briefly $(\mms,\Rmet)\sim(\mmms,\RRmet)$ or $(\mms,\Rmet)\sim_F(\mmms,\RRmet)$, if there exists a bi-Lipschitz map $F\colon\mms\to\mmms$.
\end{definition}

Quasi-isometry of Lipschitz manifolds does not depend on the chosen atlases, 
is an equivalence relation on the totaliy of LR manifolds, and preserves  topological dimensions. 

By \autoref{Le:Rev char} below --- whose easy proof employing the notions from \autoref{Sub:Metric structure} is omitted --- quasi-isometry has a convenient counterpart in terms of  pull-back metrics, a concept which is shortly outlined now (and which could easily be defined  for locally bi-Lipschitz maps with  evident changes in the statements below).

The \emph{pull-back metric} of $\RRmet$ under a bi-Lipschitz  $F\colon \mms\to\mmms$ is the family $F^*\RRmet := (F^*\RRmet_\alpha)_{\alpha\in A}$ of scalar products $F^*\RRmet_\alpha\colon \Phi_\alpha(U_\alpha) \to \smash{\bigsqcup_{z\in\Phi_\alpha(U_\alpha)}(\R^d)^{*\otimes 2}}$ given by
\begin{align}\label{Eq:Pullback metric def}
F^*\RRmet_\alpha = F_{\alpha,\beta}^*\,\RRmet_\beta\quad\Leb^d\text{-a.e.}\quad\text{on }\Phi_\alpha(U_\alpha\cap F^{-1}(V_\beta))
\end{align}
for every $\alpha\in A$ and $\beta\in B$ with $F(U_\alpha)\cap V_\beta\neq \emptyset$. Indeed, \eqref{Eq:Compatibility condition} for $\RRmet$ and the chain rule ensure that $F^*\RRmet$ is well-defined in the above way.  However,  $F^*\RRmet$ itself does not necessarily satisfy   \eqref{Eq:Compatibility condition}.


\begin{lemma}\label{Le:Rev char} The following properties hold.
\begin{enumerate}[label=\textnormal{\textcolor{black}{(}\roman*\textcolor{black}{)}}]
\item Suppose that $(\mms,\Rmet)\sim_F(\mmms,\RRmet)$ with some bi-Lipschitz $F\colon \mms\to\mmms$. Then there exists a constant $C\geq 1$ such that for every $\alpha\in A$,
\begin{align}\label{Eq:CF prop}
C^{-1}\,\Rmet_\alpha \leq F^*\RRmet_\alpha \leq C\,\Rmet_\alpha\quad\Leb^d\text{-a.e.}\quad\text{on }\Phi_\alpha(U_\alpha)
\end{align}
in the sense of quadratic forms.
\item Conversely, if there exists a locally bi-Lipschitz map $F\colon \mms\to\mmms$, a constant $C \geq 1$ and a sub-atlas $\smash{\calU' := (U_{\alpha'},\Phi_{\alpha'})_{\alpha'\in A'}}$ of $\calU$ such that  \eqref{Eq:CF prop} holds for every $\alpha'\in A'$ instead of $\alpha$, then $(\mms,\Rmet)\sim_F(\mmms,\RRmet)$.
\end{enumerate}
\end{lemma}

\begin{remark}\label{Re:Volume comparison}
By \autoref{Le:Rev char} and \eqref{Eq:Volg}, it follows in particular that if $(\mms,\Rmet)\sim_F(\mmms,\RRmet)$, there exists a constant $C \geq 1$ such that for every $x\in\mms$ and every $r>0$,
\begin{align*}
C^{-1}\,\vol_\Rmet\big[\Ball_{r/C}^\Rmet(x)\big] \leq \vol_\RRmet\big[\Ball_r^\RRmet(F(x))\big] \leq C\,\vol_\Rmet\big[\Ball_{Cr}^\Rmet(x)\big].
\end{align*}
\end{remark}

According to \autoref{Le:Rev char}, the setting of a smooth Riemannian manifold with two metrics that are uniformly elliptic to each other,  studied in \cite{saloffcoste1992}, falls into the category of quasi-isometry according to \autoref{Def:QI}.

\subsection{Some potential theory}\label{Sub:Some potential theory} Finally, we outline how Dirichlet spaces over  an LR manifold $(\mms,\Rmet)$  can be constructed. We refer the reader to \cite{davies1989,norris1997,sturm1995, sturm1996} for details, and to \cite{bouleau1991,chen2012} for general Dirichlet form theory.

To relax notation, in the sequel we do not make explicit the dependencies of all introduced objects on $\Rmet$, unless required.

Every statement in this section holds with evident modifications for $\vol$ replaced by a measure  $\meas$ as in \autoref{Sub:Volume}, except \autoref{Le:Quasi-iso Dirichlet}, as underlined in \autoref{Re:Quasi-iso Dirichlet}. 

\subsubsection{Weak derivatives}\label{Sub:Weak}  Let $W^{1,2}(\mms)$ be the space of all $f\in\Ell^2(\mms)$ whose chartwise distributional differentials $\rmd_\alpha f$, $\alpha\in A$, are $1$-forms, locally $\Ell^2$ w.r.t.~$\smash{\Leb^d}$ on $\Phi_\alpha(U_\alpha)$, which in turn give rise to a $1$-form $\rmd f :=(\rmd f_\alpha)_{\alpha\in A}$ on $\mms$ such that $\vert\rmd f\vert\in\Ell^2(\mms)$. The space $W^{1,2}(\mms)$ is complete and separable w.r.t.~the norm $\Vert\cdot\Vert_{W^{1,2}(\mms)}$ given by
\begin{align*}
\big\Vert f\big\Vert_{W^{1,2}(\mms)}^2  := \big\Vert f\big\Vert_{\Ell^2(\mms)}^2 + \big\Vert \vert \rmd f\vert\big\Vert_{\Ell^2(\mms)}^2,
\end{align*}
and $\rmd$, the so-called \emph{differential}, is a closed operator on $\Ell^2(\mms)$.  

The latter object coincides with the one introduced in \autoref{Sub:Lip functions}  on the space  $W^{1,2}(\mms)\cap\Lip_\loc(\mms)$. In particular, $W^{1,2}(\mms) \cap \Lip_\loc(\mms)$ is dense in $W^{1,2}(\mms)$, and so is  $W^{1,2}(\mms)\cap\Lip(\mms)$ by partition of unity.

\subsubsection{Dirichlet form}\label{Sub:Dir form}

The quadratic form $\Ch\colon W^{1,2}(\mms)\to [0,\infty)$ defined by
\begin{align}\label{Eq:E definition}
\Ch(f) := \frac{1}{2}\int_\mms \vert\rmd f\vert^2\d\vol
\end{align}
is a strongly local, regular Dirichlet form with domain $\F:= W^{1,2}(\mms)$ and carré du champ $\vert\rmd\cdot\vert^2\colon W^{1,2}(\mms)\to\Ell^1(\mms)$. (The latter is the same quantity for every choice of $\vol$-locally equivalent reference measure $\meas$ according to \autoref{Sub:Volume}.) Both $W^{1,2}(\mms)\cap\Lip_\loc(\mms)$ and $W^{1,2}(\mms)\cap\Lip(\mms)$ are cores for $\Ch$. 

We use the non-relabeled symbol $\Ch$ for the polarization $\smash{\Ch\colon W^{1,2}(\mms)^2\to \R}$ of \eqref{Eq:E definition}, and we do so analogously for any other quadratic form in the sequel.

The following lemma will be useful in \autoref{Sub:Quasi isom} below. It easily follows from the definition \eqref{Eq:Pullback metric def} of the pullback metric,  \autoref{Le:Rev char} and the above mentioned density of Lipschitz functions in energy.

\begin{lemma}\label{Le:Quasi-iso Dirichlet} Suppose that $(\mms,\Rmet)\sim_F(\mmms,\RRmet)$ through a bi-Lipschitz map $F\colon \mms\to\mmms$. Then   $f\in \F_{\Rmet}$ if and only if $f\circ F^{-1}\in \F_{\RRmet}$, and there exists a constant $C\geq 1$ such that for every such $f$,
\begin{align*}
C^{-1}\,\Ch_{\Rmet}(f) \leq \Ch_{\RRmet}(f\circ F^{-1})\leq C\,\Ch_{\Rmet}(f).
\end{align*}
\end{lemma}

\begin{remark}\label{Re:Quasi-iso Dirichlet} If one starts with Dirichlet forms $\smash{(\Ch_{\Rmet,\meas},\F_{\Rmet,\meas})}$ on $\mms$ and $\smash{(\Ch_{\RRmet,\mmeas},\F_{\RRmet,\mmeas})}$ on $\mmms$ induced in the evident way as above by  two Borel measures $\meas$ on $\mms$ and $\mmeas$ on $\mmms$ which are locally equivalent to the respective volume measures, to obtain the conclusion of \autoref{Le:Quasi-iso Dirichlet} one has to additionally assume that  $\meas$  and $(F^{-1})_\push\mmeas$ are mutually equivalent with bounded densities.
\end{remark}

\subsubsection{Laplacian and heat flow} Let $\Delta/2$ be the closed, self-adjoint and nonpositive generator with dense  domain $\Dom(\Delta)\subset \Ell^2(\mms)$ generating $\Ch$. 

By the spectral theorem, given any $f\in\Ell^2(\mms)$ the assignment
\begin{align*}
\ChHeat_tf := \rme^{\Delta t/2}f
\end{align*}
gives rise to the unique global solution $u(t) := \ChHeat_tf$ of the heat equation --- or, in other words, of the parabolic operator $\partial/\partial t - \Delta/2$, cf.~\cite[p.~285]{sturm1995} for more precise definitions --- on $\mms$ with initial condition $f$. The \emph{heat flow} $(\ChHeat_t)_{t\geq 0}$ constitutes a strongly continuous semigroup of linear, positivity-preserving and sub-Markovian contraction operators on $\Ell^2(\mms)$. $(\ChHeat_t)_{t\geq 0}$ extends to a semigroup of linear contraction operators on $\Ell^p(\mms)$, $p\in [1,\infty]$,  which  is  strongly continuous if $p<\infty$ and weakly$^*$ continuous provided $p=\infty$.

\subsubsection{Heat kernel}\label{Sub:Heat kernel} The heat operator $\ChHeat_t$, $t\geq 0$, can be represented as an integral operator w.r.t.~$\vol$ on $\Ell^p(\mms)$, $p\in[1,\infty]$, through a (spacially symmetric) fundamental solution $\sfp\colon (0,\infty)\times \mms^2\to (0,\infty)$, the  \emph{heat kernel}, of the parabolic operator $\partial/\partial t - \Delta/2$. Indeed, $\vol$ is locally doubling and the Dirichlet space $(\mms,\Ch,\vol)$ obeys a local Sobolev inequality, which are both Euclidean properties that transfer chartwise back to $\mms$ \cite{saloffcoste2002,sturm1995,sturm1996}, compare with  \autoref{Sub:Quasi isom} below. Moreover \cite[Thm.~3.6]{norris1997}, the metric $\met$ introduced in  \autoref{Sub:Metric structure} coincides with the one generated by the \emph{intrinsic distance} induced by $\vert\rmd\cdot\vert^2$, i.e.~for every $x,y\in\mms$,
\begin{align*}
\met(x,y) = \sup\!\big\lbrace \phi(x) - \phi(y) : \phi\in\Lip_\loc(\mms),\ \vert\rmd \phi\vert\leq 1\ \vol\text{-a.e.}\big\rbrace.
\end{align*}
The general result \cite[Prop.~2.3]{sturm1995} thus applies and leads to the above heat kernel existence. (However, the intrinsic distance does generally \emph{not} determine $\Rmet$, hence $\Ch$, unless the coefficients of $\Rmet$ are a priori known to be continuous \cite{sturm1997}.)

The map $\sfp$ may and will be chosen to be (jointly) locally Hölder \cite{sturm1996}. 
Denoting by $\lambda\geq 0$ the least eigenvalue of $-\Delta/2$, we have the following Gaussian heat kernel estimate \cite[Thm.~2.4]{sturm1995} (whose polynomial terms are absorbed by the exponentials up to the errors $\varepsilon,\delta > 0$ according to the remark at \cite[p.~309]{sturm1995}).

\begin{proposition}\label{Pr:HK bound general} Given any $x,y\in\mms$, let $R_x,R_y>0$ such that the doubling property \cite[p.~293]{sturm1995} and the  Sobolev inequality \cite[p.~294]{sturm1995} hold on $\smash{\Ball_{R_x}(x)}$ and $\smash{\Ball_{R_y}(y)}$, respectively. Define $R := \min\{R_x,R_y\}$. Then for every $\varepsilon,\delta >0$ there exists a constant $C>0$ such that for every $t\in (0,R^2]$,
\begin{align*}
\sfp(t,x,y) \leq C\,\vol\big[\Ball_{\sqrt{t}}(x)\big]^{-1/2}\,\vol\big[\Ball_{\sqrt{t}}(y)\big]^{-1/2}\,\exp\!\Big[\!-\!\frac{\met^2(x,y)}{(4+\varepsilon)t} - (1-\delta)\lambda t\Big].
\end{align*}
\end{proposition} 

Of course, in general the task is to find conditions under which the constant $C$ in \eqref{Eq:HK bd} is spacially, hence time uniform. See \cite{coulhon2008, davies1989, grigoryan1997, grigoryan2009} as well as the survey article \cite{saloffcoste2010} and the references therein for the vast literature about conditions for the former. For instance, in our work, see \autoref{Sub:Quasi isom}, this will be the case as soon as a local Bishop--Gromov-type property (stronger than the doubling property) and the Sobolev inequality hold with uniform constants.

\section{Heat kernel control pairs}\label{Ch:Control pairs}

In this chapter, we introduce the concept of \emph{heat kernel control pairs}, i.e.~suitable pairs of functions giving rise to heat kernel upper bounds with decoupled space and time dependencies. In \autoref{Sub:Gen ex}, we prove their existence in the entire generality of LR manifolds. To this aim, we first collect some preliminaries about local solutions to the heat equation on $\mms$, and we prove useful mean value inequalities for these. 

\subsection{Mean value inequalities} Every statement in this section, with evident modifications, still makes sense if we replace $\vol$ by a measure $\meas$ as in \autoref{Sub:Volume}.

\subsubsection{Local solutions}\label{Sub:Heat flow} We refer to \cite{sturm1995} and the references therein for details about the topics that are described in this subsection.

Let $I\subset\R$ be an open interval. 
Let $\smash{\Cont_\bounded(\overline{I};\Ell^2(\mms))}$ be the Banach space of all bounded, continuous functions $u\colon \overline{I}\to\Ell^2(\mms)$ with the evident norm. 

The space $\F(I\times\mms)$ is defined to consist of all strongly measurable $u\colon I \to \F$ such that $\smash{t\mapsto \Vert u(t)\Vert_\F}$ belongs to $\smash{\Ell^2(I,\Leb^1)}$, and the distributional time derivative $\dot{u}$ is a strongly measurable map $t\mapsto\dot{u}(t)$ from $I$ to $\F^{-1}$ for which $t\mapsto\Vert \dot{u}(t)\Vert_{\F^{-1}}$ belongs to $\smash{\Ell^2(I,\Leb^1)}$. Here $\F^{-1}$ is the dual space of $\F$. $\F(I\times\mms)$ is a Hilbert space w.r.t.~the evident norm.

An important fact we frequently use is the following, cf.~\cite[p.~285]{sturm1995}.

\begin{lemma}\label{Le:Cts ext}
	Every $u\in \F(I\times\mms)$ has a continuous $\smash{\Leb^1}$-version on $I$ with values in $\Ell^2(\mms)$ and which, w.r.t.~the uniform topology, continuously extends to $\smash{\overline{I}}$.
\end{lemma}

Without further notice, every element in $\F(I\times\mms)$ will thus be identified with its $\Leb^1$-version in $\smash{\Cont_\bounded(\overline{I};\Ell^2(\mms))}$.

Next, we pass over to local spaces. Let $U\subset\mms$ be  open, and let $\F_\loc(I\times U)$ denote the space consisting of all strongly measurable $u\colon I\to \Ell^0(U)$ such that for every relatively compact, open $O\subset\subset U$ and every open $J\subset\subset I$ there exists $u'\in \F(I\times\mms)$ with $\One_{O}\,u'=\One_{O}\,u$ on $J$.

Lastly, $\F_0(I\times U)$ consists of all $u\in \F(I\times \mms)$ such that for $\smash{\Leb^1}$-a.e.~$t\in I$, $\supp u(t)$ is compact in $U$. This means that $u(t) = 0$ on $\partial U$  for $\smash{\Leb^1}$-a.e.~$t\in I$, but not necessarily $u(\inf I) = 0$, provided $\inf I\in\R$, or $u(\sup I) = 0$, provided $\sup I \in\R$.


\begin{definition}\label{Def:DL} A function $u$ is termed \emph{local subsolution} of the heat equation on $I\times U$ if $u\in \F_\loc(I\times U)$ and for every $J\subset\subset I$ and every nonnegative $\phi\in \F_0(I\times U)$, 
	\begin{align*}
	\int_J \Ch\big(u(t),\phi(t)\big)\d t + \int_J \big\langle \dot{u}(t) \,\big\vert\,\phi(t) \big\rangle\d t \leq 0.
	\end{align*}
	If additionally, the function $-u$ is a local subsolution on $I\times U$ as well, we call $u$ \emph{local solution} of the heat equation on $I\times U$.
\end{definition}

\subsubsection{$L^2$- and $L^1$-mean value inequalities} With this  concept of local subsolutions, we now state and prove two mean value inequalities of $\Ell^2$- and $\Ell^p$-type, $p\in [1,2]$. In their form from \autoref{Th:L^2 mean value} and \autoref{Th:L^p mean value} below and apart from global doubling assumptions, we have only been able to find these for smooth Riemannian manifolds in \cite{grigoryan1991, grigoryan2009, guneysu2017b}. The proofs in our more general case are similar, but unlike \cite{grigoryan1991, grigoryan2009, guneysu2017b} we have to take care that local solutions are not defined to be $\smash{\Cont^2}$ in general, but rather in a weak sense recorded in \autoref{Def:DL} below.

\begin{proposition}\label{Th:L^2 mean value} Let $\Ball_r(x)\subset \mms$, $x\in\mms$ and $r>0$, be a relatively compact ball. Suppose that for some $a,\caln > 0$, we have the \emph{Faber--Krahn inequality}
	\begin{align*}
	\lambda(O) \geq a\,\vol[O]^{-2/\caln}
	\end{align*}
	for every open set $O \subset \Ball_r(x)$. Then there exists a constant $C>0$ depending only on $n$ such that for every $t>0$, every jointly continuous local subsolution $u$ of the heat equation on $(0,t)\times \Ball_r(x)$, and every $\varepsilon \in (0,t/4)$,
	\begin{align*}
	u_+^2(t-\varepsilon)(x) \leq \frac{C\,a^{-\caln/2}}{\min\{\sqrt{t-2\varepsilon},r\}^{\caln+2}}\int_\varepsilon^{t-\varepsilon}\!\!\int_{\Ball_{r}(x)} u_+^2(s)\d\vol\d s.
	\end{align*}
\end{proposition}

\begin{remark} Note that, unlike the expression $u(t-\varepsilon)$, $\varepsilon \in (0,t/4)$, in \autoref{Th:L^2 mean value}, the evaluation $u(t)$ does not need to make sense in general (in contrast to the $\smash{\Cont^2}$-notion of \cite[Thm.~15.1]{grigoryan2009}) by definition of $\F_\loc((0,t)\times\Ball_r(x))$. 
	
For the same reason, it might a priori happen that $\smash{\Vert u_+\Vert_{\Ell^2(\Ball_r(x))}}$ is not square-integrable on $(\varepsilon, t-\varepsilon)$ even for $\varepsilon =0$. However, these degeneracy issues can always be circumvented by shrinking the radius $r$.
\end{remark}

\begin{proof}[Proof of \autoref{Th:L^2 mean value}] In terms  of local subsolutions as in \autoref{Def:DL}, the smooth argument for  \cite[Thm.~15.1]{grigoryan2009} carries over without essential changes in the following setup. Consider the  function $v\in \F_\loc((-\varepsilon/2,t-3\varepsilon/2)\times\Ball_{r-\iota}(x))$ with
	\begin{align*}
	v(s) := u(s+\varepsilon),
	\end{align*}
	where $\iota\in (0,r)$. In fact, the map $s\mapsto \Vert v(s)\Vert_{\Ell^2(\Ball_{r-\iota}(x))}$ belongs to $\Ell^2((0,t-2\varepsilon),\Leb^1)$, and $v$ has a well-defined value at $t-2\varepsilon$ by \autoref{Le:Cts ext}. Clearly, $v$ is a local solution to the heat equation on $(-\varepsilon/2,t-3\varepsilon/2)\times \Ball_{r-\iota}(x)$. Even better, the shifting of $u$ by $\varepsilon$ allows us to take $J := (0,t-2\varepsilon)$ as test interval in  \autoref{Def:DL}. It therefore remains to follow the lines in \cite[Sec.~15.1]{grigoryan2009} with $u$ replaced by $v$, $T$ replaced by $t-2\varepsilon$, and $R$ replaced by $r-\iota$ therein. This entails the existence of a constant $C>0$ depending only on $\caln$ such that
	\begin{align*}
	v_+^2(t-2\varepsilon)(x) \leq \frac{C\,a^{-\caln/2}}{\min\{\sqrt{t-2\varepsilon}, r-\iota\}^{\caln+2}}\int_0^{t-2\varepsilon}\!\!\int_{\Ball_{r-\iota}(x)} v_+^2(s)\d\vol\d s,
	\end{align*}
	which directly provides the asserted inequality by sending $\iota \to 0$.
\end{proof}

\begin{theorem}\label{Th:L^p mean value} For every $\caln>0$, there exists a constant $C>0$ with the following property. For every $x\in\mms$ and every $r>0$ such that $\Ball_r(x)\subset\mms$ is relatively compact and for which there exists $a>0$ such that
	\begin{align*}
	\lambda_1(O) \geq a\,\vol[O]^{-2/\caln}
	\end{align*}
	for every open $O\subset\Ball_r(x)$, for every $\tau,t>0$ with $\tau \leq r^2$ and $\tau < t$, every jointly continuous nonnegative local subsolution $u$ of the heat equation on $\smash{(t-\tau,t)\times\Ball_{\sqrt{\tau}}(x)}$ such that $\smash{u\in\F_\loc((t-\tau-\zeta,t+\zeta)\times \Ball_{\sqrt{\tau}+\zeta}(x))}$ for some error $\zeta\in (0,t -\tau)$, and every $q\in [1,2]$, 
	\begin{align*}
	u^q(t)(x) \leq \frac{C\,a^{-\caln/2}}{\tau^{1+\caln/2}}\int_{t-\tau}^t\!\int_{\Ball_{\sqrt{\tau}}(x)} u^q(s)\d\vol\d s.
	\end{align*} 
\end{theorem}

\begin{proof} We apply \autoref{Th:L^2 mean value} to the radius $r := \sqrt{\tau}$ and to the local subsolution $\smash{v\in \F_\loc((0,\tau)\times\Ball_{\sqrt{\tau}-\iota}(x))}$ to the heat equation on $\smash{(0,\tau)\times\Ball_{\sqrt{\tau}-\iota}(x)}$, with
	\begin{align*}
	v(s) := u(t-\tau + s),
	\end{align*}
	where $\iota\in (0,\sqrt{\tau})$. Given any $\varepsilon\in (0,\tau/4)$, we thus get
	\begin{align*}
	v^2(\tau-\varepsilon)(x) \leq \frac{C\,a^{-\caln/2}}{\min\{\sqrt{\tau-2\varepsilon},\sqrt{\tau}-\iota\}^{\caln+2}}\int_\varepsilon^{\tau-\varepsilon}\!\!\int_{\Ball_{\sqrt{\tau}-\iota}(x)} v^2(s)\d\vol\d s
	\end{align*}
	for some universal constant $C>0$ depending only on $\caln$. By a change of variables on the r.h.s., employing that by assumption, the function $\smash{s\mapsto\Vert u(s)\Vert_{\Ell^2(\Ball_{\sqrt{\tau}}(x))}}$ belongs to $\smash{\Ell^2((t-\tau,t),\Leb^1)}$ and that $u(t-\varepsilon)(x) \to u(t)(x)$ as $\varepsilon\to 0$, and sending  $\varepsilon\to 0$ and $\iota \to 0$, we obtain
	\begin{align*}
	u^2(t)(x) \leq \frac{C\,a^{-\caln/2}}{\tau^{1+\caln/2}}\int_{t-\tau}^t\!\int_{\Ball_{\sqrt{\tau}}(x)} u^2(s)\d\vol\d s.
	\end{align*} 
	
	From here on, Li--Wang's $\Ell^2$-to-$\Ell^q$ reduction procedure \cite{li1999} from the proofs of \cite[Thm.~IV.15]{guneysu2017} and \cite[Prop.~2.10]{guneysu2017b} applies \emph{verbatim} and gives the claim.
\end{proof}

\subsection{Control pairs and their general existence}\label{Sub:Gen ex} 

\subsubsection{Basic definition} The following is motivated by \cite[Def.~IV.16]{guneysu2017}.

\begin{definition}\label{Def:Heat kernel control pair} An ordered pair $(\Xi, \Sigma)$ is termed a \emph{heat kernel control pair} for $(\mms,\Rmet)$ if
\begin{enumerate}[label={\textnormal{\alph*.}}]
\item $\Xi \colon\mms\to (0,\infty)$ is continuous and bounded away from zero,
\item $\Sigma\colon (0,1]\to (0,\infty)$ is Borel measurable,
\item there exists a constant $C>0$ such that for every $x\in\mms$ and every $t\in (0,1]$,
\begin{align}\label{Eq:Heat kernel bound control pair}
\sup_{y\in\mms} \sfp(t,x,y) \leq C\,\Sigma(t)\,\Xi(x),
\end{align}
\item\label{La:Integr caln} for every  $p \in [1, \infty)$ provided $d =1$ and every $p\in (d/2, \infty)$ provided $d\geq 2$, there exists a constant $C>0$ such that
\begin{align*}
\int_0^\infty \!\Sigma^{1/p}(t)\,\rme^{-Ct}\d t < \infty.
\end{align*}
\end{enumerate}
\end{definition}

\autoref{Def:Heat kernel control pair} is still meaningful on any quasi-regular, strongly local Dirichlet space $(\mms,\Ch,\meas)$ with heat kernel --- in this case, we will usually speak about heat kernel control pairs for $(\mms,\Ch,\meas)$. There are many known sufficient conditions for these to admit control pairs, compare with \autoref{Sub:Quasi isom}. On the other hand, our main result from \autoref{Prop:heat kernel bound} below --- namely that $\mms$ \emph{always} admits a canonical heat kernel control pair --- really uses the (Lipschitz) manifold structure. 


\subsubsection{The Euclidean radius}\label{Subsub:Eucl Rad} We introduce the following concept which extends the smooth treatise from  \cite[Def.~IV.12]{guneysu2017}. The idea behind it, in view of \autoref{Prop:heat kernel bound},  is to transfer the global constants from the Euclidean Faber--Krahn inequality  to $\mms$, paying the price of possibly non-uniformity of the size of the considered subsets of $\mms$ on which the latter holds.

Compared to \cite{guneysu2017}, additional care has to be taken to the possible presence of a boundary of $\mms$.

\begin{definition}\label{Def:Eucl rad} Given any $x\in\mms$ and  $b>1$, we define $r_\rmE(x,b)$ as the sup\-remum over all radii $r>0$ such that
\begin{enumerate}[label=\textnormal{\alph*.}]
\item $\Ball_r(x) \subset \mms$ is relatively compact, and
\item\label{La:Chart b} there exists  a chart $(U,\Phi)$ around $x$ such that $\Ball_r(x) \subset U$, $\smash{\Phi\big\vert_{\Ball_r(x)}}$ is a homeomorphism to an open subset $\Phi(\Ball_r(x))$ of $\smash{\R^d}$ or $\smash{\R^{d-1}\times [0,\infty)}$, and for every $\xi\in \smash{\R^d}$,
\begin{align}\label{Eq:Quasi-isom}
b^{-1}\,\vert\xi\vert^2 \leq \Rmet^U(\xi,\xi) \leq b\,\vert\xi\vert^2 \quad\Leb^d\text{-a.e.}\quad\text{on }\Phi(\Ball_r(x)),
\end{align}
where $\smash{\Rmet^U\colon \Phi(U) \to \bigsqcup_{z\in \Phi(U)}(\R^d)^{*\otimes 2}}$ is the Riemannian metric on the image of the chart $(U,\Phi)$ according to \autoref{Def:Riem metrics}.
\end{enumerate}
\end{definition}

\begin{remark}\label{Re:Infinity} Note that either $r_\rmE(\mms,b) \subset (0,\infty)$ or $r_\rmE(\mms,b) = \{\infty\}$ for every $b>1$. This can easily be seen from the definition, or alternatively using the $1$-Lipschitz property from \autoref{lemma:propr} below.
\end{remark}

\begin{example}\label{Ex:Spher symm} 
Let us consider the particular case of  spherically symmetric manifolds (see e.g.~\cite[Ch.~8]{coulhon1997}). Let $d\geq 2$, fix a point $\smash{o \in \R^d}$ and a positive bi-Lipschitz function $\psi\colon [0,\infty)\to [0,\infty)$ which is smooth at $0$ and satisfies
\begin{align}\label{eq:incond}
\begin{split}
\psi(0) &= 0,\\ \psi'(0) &= \text{Lip}(\psi).
\end{split}
\end{align}
We define a spherically symmetric LR manifold $(\mms_\psi,\Rmet_\psi)$ in the following way.
\begin{enumerate}[label=\alph*.]
\item As a set of points, $\mms_\psi$ is $\R^d$.
\item In polar coordinates $\smash{(r, \theta) \in (0,\infty) \times \boldsymbol{\mathrm{S}}^{d-1}}$ at $o$  the Riemannian metric $\Rmet_\psi$ on $\mms_\psi\setminus \{o\}$ is defined as
\begin{equation}\label{eq:sphsym}
\rmd s^2 = \rmd r^2 + \psi^2(r) \d \theta^2,
\end{equation}
where $\rmd \theta^2$ denotes the standard Riemannian metric on $\smash{\boldsymbol{\mathrm{S}}^{d-1}}$.
\item The Riemannian metric $\Rmet_\psi$ at $o$ is a smooth extension of \eqref{eq:sphsym}, whose existence is ensured by \eqref{eq:incond}.
\end{enumerate}
Since  $\psi$ is bi-Lipschitz, we have in particular that
\[
\rmd r^2 + \text{Lip}(\psi)^{-2}\, r^2 \d \theta^2 \le \rmd r^2 + \psi^2(r) \d \theta^2 \le \rmd r^2 + \text{Lip}(\psi)^2\, r^2 \d \theta^2.
\]
Recalling that $\d r^2 + r^2 \d \theta^2$ is the standard Euclidean metric of $\R^d$, the above chain of inequalities ensures that, for any $b > \text{Lip}(\psi)$, it holds $r_\rmE(o, b) = \infty$. Therefore, by the previous \autoref{Re:Infinity} we have $r_\rmE(\mms_\psi, b) = \{\infty\}$ for every $b>\Lip(\psi)$.
\end{example}

\begin{lemma}\label{lemma:propr} For every $x\in\mms$ and for every $b > 1$, the following properties hold.
\begin{enumerate}[label=\textnormal{\textcolor{black}{(}\roman*\textcolor{black}{)}}]
\item\label{La:i} We have $r_\rmE(x,b) \in (0, \infty]$.
\item\label{La:r_E properties} For every $\varepsilon > 0$, the function $\smash{x\mapsto \min\{r_\rmE(x,b), \varepsilon\}}$ on $\mms$ is $1$-Lipschitz w.r.t.~$\met$. In particular,  for every compact set $K\subset\mms$,
\begin{align*}
\inf_{x\in K} r_\rmE(x,b) > 0.
\end{align*}
\end{enumerate}
\end{lemma}

\begin{proof} To prove \ref{La:i}, we first choose $r>0$ such that $\Ball_r(x) \subset U_\alpha$ is relatively compact for some fixed chart $(U_\alpha,\Phi_{\alpha})$, $\alpha\in A$. According to Remark \ref{rmk:lambda}, for this coordinate system we have the validity of \eqref{Eq:Quasi-isom}, where $b$ is replaced by a suitable $\lambda_\alpha \ge 1$. If  $\lambda_\alpha \le b$, \eqref{Eq:Quasi-isom} is automatically satisfied also for the given $b$, while in the case  $\lambda_\alpha \ge b$ we have to rescale the coordinate system. This is done by scaling $\smash{\Rmet_\alpha^{-1/2}}$ in such a way that its eigenvalues range between $\smash{\sqrt{b/\lambda_\alpha}}$ and $\smash{\sqrt{\lambda_\alpha/b}}$, namely multiplying those eigenvalues of $\smash{\Rmet_\alpha^{-1/2}}$ which are greater than 1 by the factor $\smash{1/\sqrt{b}}$, and those which are smaller than 1 by the factor $\smash{\sqrt{b}}$ . In such a way, we obtain a coordinate system for which \eqref{Eq:Quasi-isom} is satisfied, thus $0 < r \leq r_\rmE(x,b)$.


Next, we prove \ref{La:r_E properties}. To simplify notation, let us set $r(x) := r_\rmE(x,b)$ and $r_\varepsilon (x) := \min \{ r(x), \varepsilon  \}$,  $x \in \mms $ and $\varepsilon > 0$. It clearly suffices to prove the $1$-Lipschitz continuity of the function $x \mapsto r_\varepsilon(x)$ w.r.t.~the distance $\met$ defined in Section \ref{Sub:Metric structure}.

Let us first assume that $\smash{y \in \Ball_{r_\varepsilon(x)}(x)}$. By definition of $r(y)$, we have
\begin{align*}
	r(y) \geq r(x) - \met(x,y) \geq r_\varepsilon(x) - \met(x,y).
\end{align*}
Since $\met(x,y) < r_\varepsilon(x) \leq \varepsilon$, it also follows that
\begin{align*}
0 < r_\varepsilon(x) - \met(x,y) \leq \varepsilon,
\end{align*}
and in particular $r_\varepsilon(y) \geq r_\varepsilon(x) - \met(x,y)$. 
In the case $r_\varepsilon(x) \ge r_\varepsilon(y)$, we already get
\begin{align}\label{Eq:1Lip}
| r_\varepsilon(x) - r_\varepsilon(y) | \le \met(x, y).
\end{align}
Otherwise, if $r_\varepsilon(x) < r_\varepsilon(y)$ we have $x \in \Ball_{r_\varepsilon(y)}(y)$ and, as above, we derive the inequalities $r(x) \ge r_\varepsilon(y) - \met(x, y)$ and hence $r_\varepsilon(x) \ge r_\varepsilon(y) - \met(x, y)$, which is the remaining part to prove \eqref{Eq:1Lip}.

Now, let us assume that $y \notin \Ball_{r_\varepsilon(x)}(x)$, so that
\begin{align*}
	\met(x,y) \geq r_\varepsilon(x) \geq r_\varepsilon(x) - r_\varepsilon(y).
\end{align*} 
From this inequality, \eqref{Eq:1Lip} directly follows if $\smash{x\notin \Ball_{r_\varepsilon(y)}(y)}$ as well. Otherwise, if $\smash{x \in \Ball_{r_\varepsilon(y)}(y)}$ we can argue as above to prove that $r_\varepsilon(x) \ge r_\varepsilon(y) - \met(x, y)$. Since in this case $r_\varepsilon(y) > r_\varepsilon(x)$, we finally obtain \eqref{Eq:1Lip}.
\end{proof}

\begin{remark} With some technical effort employing the notions from \autoref{Sub:Metric structure}, following the lines of \cite[pp.~59--60]{guneysu2017} one can show the following. Given any $x\in\mms$, $b>1$ and $r \in (0,r_\rmE(x,b))$, we have the inclusions
	\begin{align*}
	B_{r/\sqrt{b}}(0) \subset \Phi(\Ball_r(x)) \subset B_{r\sqrt{b}}(0), 
	\end{align*}
	where $B$ denotes the Euclidean ball in $\R^d$ or $\R^{d-1}\times [0,\infty)$ (depending on whether $\Ball_r(x) \cap\partial\mms = \emptyset$ or not), and $(\Phi,U)$ is a chart witnessing \autoref{La:Chart b} w.r.t.~the data $x$ and $b$ in \autoref{Def:Eucl rad}. Moreover, we have
	\begin{align*}
	\met(x, z) \le \sqrt{b}\, \vert \Phi(z) - \Phi(x) \vert
	\end{align*}
	for every $\smash{z \in \Phi^{-1}(B_{r/\sqrt{b}}(0))}$, while for every $z \in \Ball_r(x)$,
	\begin{align*}
	| \Phi(z) -\Phi(x)| \le \sqrt{b}\, \met(x, z).
	\end{align*}
	These properties, however, are not needed in the proof of \autoref{Prop:heat kernel bound} below.
\end{remark}

\subsubsection{Main result and construction of the control pair} Based upon the Euclidean radius $r_\rmE$ from \autoref{Def:Eucl rad}, we now construct the desired general control pair.

We first establish the following local Faber--Krahn inequality with uniform constants. Let $\lambda_1(O)$ and $\mu_1(\Omega)$ denote the first eigenvalues of the Dirichlet Laplacians on given relatively compact, open domains $O\subset M$ and $\Omega\subset\smash{\R^d}$, respectively. Recall that by the Euclidean Faber--Krahn theorem \cite[p.~367]{grigoryan2009}, there exists a constant $c>0$ depending only on $d$ such that for every nonempty $\Omega$ as above,
\begin{align}\label{Eq:FK Eucl}
	\mu_1(\Omega) \geq c\,\Leb^d[\Omega]^{-2/d}.
\end{align}

\begin{lemma}\label{Le:FKM}
Given any $b>1$, $\varepsilon_1>0$ and $\varepsilon_2 > 1$, define $R\colon \mms\to (0,\varepsilon_1/\varepsilon_2)$ by
\begin{align}\label{Eq:R def}
	R(y) := \min\{r_\rmE(y,b),\varepsilon_1\}/\varepsilon_2.
\end{align}
Then there exists a constant $a>0$ depending only on $b$ and $d$ such that for every $x\in \mms$, $\Ball_{R(x)}(x)$ is relatively compact, and for every nonempty open $O\subset \Ball_{R(x)}(x)$,
 \begin{align*}
 	\lambda_1(O) \geq a\,\vol[O]^{-2/d}.
 \end{align*}
\end{lemma}

\begin{proof} We first assume that $O \cap \partial\mms \neq \emptyset$. Given any $\varepsilon > 0$, let $f\in \Lip(O)$ be compactly supported in $O$ with $0 < \Vert f\Vert_{\Ell^2(O)} \leq 1$ and 
	\begin{align*}
		\lambda_1(O) &\geq \int_O \vert \rmd f\vert^2\d\vol - \varepsilon.
	\end{align*}
	Let $(U,\Phi)$ be a chart witnessing \autoref{La:Chart b} in \autoref{Def:Eucl rad} according to the definition of $R$ and $r_\rmE$. The parametrization $\smash{f^\diamond := f\circ \Phi^{-1}\big\vert_{\Phi(O)}\colon\Phi(O)\to\R}$ is Lipschitz and has compact support in $\Phi(O) \subset \smash{\R^{d-1}\times [0,\infty)}$ which possibly includes  $\smash{\R^{d-1}\times\{0\}}$. Therefore, successively using \eqref{Eq:Volg}, \eqref{Eq:Ptw norm def} and \eqref{Eq:Quasi-isom}, we obtain the existence of a constant $\eta > 0$ depending only on $b$ and $d$ such that 
	\begin{align}\label{Eq:eta1}
	\lambda_1(O) \geq \eta\int_{\Phi(O)} \big\vert \rmd f^\diamond\big\vert_{\Rmet^\rmE}^2\d\Leb^d - \varepsilon. 
	\end{align}
	Observe that, however, $\Phi(O)$ is not open in $\smash{\R^d}$, whence we cannot apply  \eqref{Eq:FK Eucl} to \eqref{Eq:eta1}. We bypass this  by a reflection technique. Let $\smash{\rho \colon \R^{d-1}\times [0,\infty) \to \R^d}$ be the map which precisely  flips the sign of the $d$-th coordinate of its argument, and note that $\Omega := O \cup \rho(O)$ is open and relatively compact. Define the reflection $\smash{f_\rmr^\diamond}\colon \Omega\to \R$ of the function $\smash{f^\diamond}$ at $\smash{\R^{d-1}\times [0,\infty)}$ by $\smash{f_\rmr^\diamond(x_1,\dots,x_{d-1},x_d)} := \smash{f^\diamond(x_1,\dots,x_{d-1}, \vert x_d\vert)}$. Clearly, $\smash{f_\rmr^\diamond}$ is Lipschitz and compactly supported in $\Omega$, thus   belongs to the form domain of the Dirichlet Laplacian on $\Omega$. Hence, since 
	\begin{align*}
		\big\Vert f_\rmr^\diamond\big\Vert_{\Ell^2(\Omega,\Leb^d)} &\asymp \Vert f\Vert_{\Ell^2(O)},\\
		\big\Vert \big\vert \rmd f_\rmr^\diamond\big\vert_{\Rmet^\rmE}\big\Vert_{\Ell^2(\Omega,\Leb^d)} &\asymp \big\Vert \big\vert \rmd f^\diamond\big\vert_{\Rmet^\rmE}\big\Vert_{\Ell^2(\Phi(O),\Leb^d)},
	\end{align*}
	up to uniform constants depending only on $b$ and $d$, \eqref{Eq:eta1} and  \eqref{Eq:FK Eucl} yield
	\begin{align*}
		\lambda_1(O) \geq \eta\,\mu_1[\Omega] -\varepsilon \geq  \eta\,\Leb^d[\Omega]^{-2/d} - \varepsilon \geq \eta\,\Leb^d[\Phi(O)]^{-2/d} - \varepsilon
	\end{align*}
	for some  $\eta > 0$ depending only on $b$ and $d$ which, as customary, is allowed to change from left to right. Employing \eqref{Eq:Volg} and \eqref{Eq:Quasi-isom} again, it follows that $\smash{\Leb^d[\Phi(O)]^{-2/d}} \geq \smash{\eta\,\vol[O]^{-2/d}}$ for some constant $\eta> 0$ depending only on $b$ and $d$, which implies the claimed inequality upon letting $\varepsilon \to 0$.
	
	A similar, more straightforward argument without the above reflection technique readily covers the case $O\cap \partial\mms = \emptyset$.
\end{proof}


\begin{theorem}\label{Prop:heat kernel bound} For every $b > 1$ there exists a constant $C > 0$ depending only on $b$ and $d$ such that for every $\varepsilon_1 > 0$, every $\varepsilon_2 > 1$, every $x, y \in \mms$ and every $t>0$
\begin{align*}
\sfp(t, x, y) \le 	C\min\{ t, R^2(x) \}^{-{d}/{2}}, 
\end{align*}
where $R$ is defined as in \eqref{Eq:R def}.
\end{theorem}

\begin{proof} Given any $y\in \mms$ and any $t>0$, using \autoref{Le:FKM} we apply \autoref{Th:L^p mean value} to the following setup: $\caln := d$, $u(s) := \sfp(s,\cdot,y)$ for every $s>0$, $r := R(x)$, $t$ replaced by $t+\delta$ for some fixed $\delta > 0$,  $\tau := \min\{t,R^2(x)\}$, an appropriate $\zeta \in (0,t+\delta-\tau)$, and $q:=1$. This yields the existence of a constant $C>0$ depending only on $d$ with
	\begin{align*}
		\sfp(t+\delta,x,y) \leq \frac{C\,a^{-d/2}}{\tau^{1+d/2}}\int_{t+\delta - \tau}^{t+\delta}\int_{\Ball_{\sqrt{\tau}}(x)}  \sfp(s,\cdot,y)\d\vol\d s \leq \frac{C\,a^{-d/2}}{\tau^{d/2}}.
	\end{align*}
	The result follows after letting $\delta \to 0$.
\end{proof}

\begin{example}\label{Ex:Spher symm II} Retain the setting of \autoref{Ex:Spher symm}. Employing \autoref{Prop:heat kernel bound}, upon choosing $b>1$ appropriately we find a constant $C>0$ depending only on $d$ such that for every $\varepsilon_1 > 0$, every $\varepsilon_2 > 1$, every $x,y\in\mms_\psi$ and every $t>0$,
	\begin{align*}
		\sfp(t,x,y) \leq C\min\{t,\varepsilon_1^2\,\varepsilon_2^{-2}\}^{-d/2}.
	\end{align*}
	This resembles a variant of the smooth result \cite[Thm.~8.3]{coulhon1997}.
\end{example}

Using these results we can then prove the existence of the claimed heat kernel control pair for an LR manifold.

\begin{corollary}\label{prop:Existence control pair} Every LR  manifold $(\mms, \Rmet)$ admits a heat kernel control pair according to \autoref{Def:Heat kernel control pair}.
\end{corollary}

\begin{proof}
\autoref{Prop:heat kernel bound} ensures that  for every $b > 1$ there exists a constant $C > 0$ depending only on $b$ and $d$ such that
\[
\sfp(t, x, y) \le \dfrac{C}{t^{d/2}} +  \dfrac{C}{R^d(x)} \le \dfrac{C}{R^d(x)}\, \Big[ \dfrac{\varepsilon_1^d}{\varepsilon_2^d\, t^{d/2}} + 1 \Big].
\] 
In the last inequality, we used that $R\leq \varepsilon_1/\varepsilon_2$ on $\mms$. Since $R$ is $1/\varepsilon_2$-Lipschitz by Lemma \ref{lemma:propr}, the functions $\Xi\colon\mms\to (0,\infty)$ and $\Sigma\colon (0,1]\to (0,\infty)$ with
\begin{align*}
	\Xi(x) &:= R^{-d}(x),\\
	\Sigma(t) &:= 1+ \varepsilon_1^d\,\varepsilon_2^{-d}\, t^{-d/2}
\end{align*}
define a heat kernel control pair for $(\mms,\Rmet)$.
\end{proof}

\subsection{Improvements under uniform constants}\label{Sub:Quasi isom} If $(\mms,\Rmet)$ has better geometric properties, it also admits different explicit heat kernel control pairs than the one derived in \autoref{prop:Existence control pair}. Deducing these (following standard lines as in \cite{grigoryan2009,saloffcoste1992}) is the goal of this section. In particular, we point out that while \autoref{Prop:heat kernel bound} and \autoref{prop:Existence control pair} are really restricted to the heat kernel induced by the volume measure $\vol$, \autoref{Th:Quasi-isometry} below  works for any given $\meas$ as in \autoref{Sub:Volume}. 

In this section we assume the following  conditions.
\begin{enumerate}[label=\alph*.]
	\item\label{La:A} \textit{Uniform local doubling.} There exists a constant $C_{\mathrm{LD}}\geq 1$ such that for every $x\in\mms$ and every $s,s'>0$ with $s'\leq s$,
	\begin{align*}
	\meas[\Ball_s(x)] \leq C_{\mathrm{LD}}\,\meas[\Ball_{s'}(x)]\,(s/s')^d\,\rme^{C_{\mathrm{LD}}s}.
	\end{align*}
	\item\label{La:B} \textit{Uniform Sobolev inequality.} There exist constants $C_{\mathrm{S}}\geq 1$ and $N > 2$ such that for every $x\in\mms$, every $r>0$ and every $f\in W^{1,2}(\mms)\cap\Cont_\comp(\Ball_r(x))$,
	\begin{align*}
	&\Big[\!\int_\mms \vert f\vert^{2N/(N-2)}\d\meas\Big]^{(N-2)/N}\\
	&\qquad\qquad \leq C_{\mathrm{S}}\,\rme^{C_{\mathrm{S}}r}\,\meas[\Ball_r(x)]^{-2/N}\,r^2\int_{\mms}\big[\vert\rmd f\vert^2 + r^{-2}\,f^2\big]\d\meas.
	\end{align*}
\end{enumerate}

\begin{theorem}\label{Th:Quasi-isometry} Retain the previous assumptions, and let  $\varepsilon,\delta > 0$. Then there exists a constant $C>0$ depending only on $d$, $C_{\mathrm{LD}}$, $C_{\mathrm{S}}$, $\varepsilon$ and $\delta$ such that for every $x,y\in\mms$ and every $t\in (0,1]$,
	\begin{align*}
	\sfp(t,x,y) \leq C\,\meas[\Ball_1(x)]^{-1}\,t^{-d/2}\,\exp\!\Big[\!-\!\frac{\met^2(x,y)}{(4+\varepsilon)t} - (1-\delta)\lambda t\Big].
	\end{align*}	
\end{theorem}

Recall \autoref{Sub:Heat kernel} for the meaning of $\lambda\geq 0$. A \emph{local} estimate similar to \autoref{Th:Quasi-isometry} can be found in \cite[Thm.~2.1]{norris1997}.

Combined with \autoref{Th:Kato criterion}, \autoref{Th:Quasi-isometry} then gives the following desired  explicit heat kernel control pair.

\begin{corollary}\label{Cor:Quasi-isometry} Retain the assumptions and the notation of \autoref{Th:Quasi-isometry}, and let $d\geq 2$ as well as $p\in (d/2,\infty)$. Define $\Xi\colon \mms \to \R$ and $\Sigma \colon (0,1]\to (0,\infty)$ by 
	\begin{align*}
	\Xi(x) &:= \meas[\Ball_1(x)]^{-1},\\
	\Sigma(t) &:= t^{-d/2}.
	\end{align*}
	Then $(\Xi,\Sigma)$ is a heat kernel control pair for $(\mms,\Rmet)$. 
\end{corollary}

\begin{proof}[Proof of \autoref{Th:Quasi-isometry}] By \ref{La:A}~applied to $s := 2s'$, $s' \in (0,1/2)$,  the doubling property according to \cite[pp.~293--294]{sturm1995} holds on every ball $B\subset\mms$ of radius no larger than $1$ with doubling constant $\smash{2^{N(B)}}$, where $N(B) := d + \log_2(C_{\mathrm{LD}}\,\rme^{C_{\mathrm{LD}}})$ in the notation of \cite{sturm1995}. By \ref{La:B}, the Sobolev inequality according to \cite[pp.~294--295]{sturm1995} also holds on every ball $B\subset\mms$ of radius no larger than $1$ with constant $C_S(B) := C_{\mathrm{S}}\,\rme^{\mathrm{S}}$. \autoref{Pr:HK bound general}  ensures that for every $\varepsilon,\delta > 0$ there exists a constant $C>0$ depending only on $N(B)$ and $C_\rmS(B)$ such that for every $t\in (0,1]$,
	\begin{align}\label{Eq:HK bd}
	\sfp(t,x,y) \leq C\,\meas\big[\Ball_{\sqrt{t}}(x)\big]^{-1/2}\,\meas\big[\Ball_{\sqrt{t}}(y)\big]^{-1/2}\, \exp\!\Big[\!-\! \frac{\met^2(x,y)}{(4+\varepsilon/2)t}-(1-\delta)\lambda t\Big].
	\end{align} 
	
	Next, we get rid of the term $\smash{\meas[\Ball_{\sqrt{t}}(y)]^{-1/2}}$ in \eqref{Eq:HK bd}. Using \ref{La:A}, we get
	\begin{align*}
	\meas\big[\Ball_{\sqrt{t}}(x)\big] &\leq \meas\big[\Ball_{\met(x,y) + \sqrt{t}}(y)\big]\\
	&\leq \rme\,C_{\mathrm{LD}}\,\meas\big[\Ball_{\sqrt{t}}(y)\big]\,\Big[1+\frac{\met_\Rmet(x,y)}{\sqrt{t}}\Big]^{d}\,\rme^{C_{\mathrm{LD}}\met(x,y)}
	\end{align*}
	for every $x,y\in\mms$ and every $t\in (0,1]$. 
	After  absorbing the latter polynomial and the latter exponential, whose exponent only depends linearly on $\smash{\met(x,y)}$, into the exponential in \eqref{Eq:HK bd}, we infer the existence of a constant $C > 0$ depending  only on $d$, $C_{\mathrm{LD}}$, $C_{\mathrm{S}}$, $\varepsilon$ and $\delta$ such that for every $x,y\in\mms$ and every $t\in (0,1]$,
	\begin{align}\label{Eq:HK bound II}
	\sfp(t,x,y) \leq C\,\meas\big[\Ball_{\sqrt{t}}(x)\big]^{-1}\,\exp\!\Big[\!-\!\frac{\met^2(x,y)}{(4+\varepsilon)t} - (1-\delta)\lambda t\Big].
	\end{align}
	
	Finally we make the radius of the ball on the r.h.s.~of \eqref{Eq:HK bound II} independent of $t$. To this aim, we again use \ref{La:A} and obtain
	\begin{align*}
	\meas[\Ball_1(x)] \leq C_{\mathrm{LD}}\,\meas\big[\Ball_{\sqrt{t}}(x)\big]\,t^{-d/2}\,\rme^{C_{\mathrm{LD}}}
	\end{align*}
	for every $x\in \mms$ and every $t\in (0,1]$. With \eqref{Eq:HK bound II}, this terminates the proof.
\end{proof}

Finally, we discuss an example in which the hypotheses of \autoref{Th:Quasi-isometry} are satisfied (recall \autoref{Sub:Quasi isom}). \autoref{Ex:QI} follows the smooth treatise \cite{braun2020}.

\begin{example}\label{Ex:QI} Let $(\mmms,\RRmet)$ be a smooth, complete Riemannian manifold for which there exists $K\geq 0$ such that $\Ric_\RRmet\geq -K$ on $\mmms$. Assume that $(\mms,\Rmet)\sim_F(\mmms,\RRmet)$ by some bi-Lipschitz $F\colon \mms\to\mmms$. Then $\smash{(\mms,\Ch_{\Rmet,\vol_\Rmet},\vol_\Rmet)}$ satisfies the hypotheses of \autoref{Th:Quasi-isometry}. Indeed, property \ref{La:A} holds true on $\mmms$ for $\meas$ replaced by $\smash{\vol_\RRmet}$ with $C_{\mathrm{LD}} := \smash{\max\{1, \sqrt{(d-1)K}\}}$ \cite[Thm.~5.6.4]{saloffcoste2002}. By \cite[Thm.~2.6]{sturm1996}, there exists a constant $C\geq 1$ depending only on $d$ and $K$ such that \ref{La:B} holds for $\smash{(\mmms,\Ch_{\Rmet,\vol_\RRmet},\vol_\RRmet)}$ with $C_{\mathrm{S}} := C$. The claim follows since \ref{La:A} and \ref{La:B} are qualitatively preserved under quasi-isometry thanks to \autoref{Le:Rev char}, \autoref{Re:Volume comparison},  \autoref{Le:Quasi-iso Dirichlet} and \autoref{Re:Quasi-iso Dirichlet}.
\end{example}

\begin{example}\label{Ex:QII} Suppose that $(\mms,\met,\meas)$, possibly with $\meas \neq\vol$, is bi-Lipschitz equivalent (as a metric measure space) to an $\RCD(K,N)$ space $(N,\met',\mmeas)$, $K\in\R$ and $N\in [1,\infty)$. For instance, $(N,\met',\mmeas)$ might be a smooth, geodesically complete Riemannian manifold with $\Ric_\RRmet\geq K$ and convex boundary \cite[Thm.~2.4, Cor.~2.6]{han2020}. Then \ref{La:A} \cite[Thm.~2.3]{sturm2006b} and \ref{La:B} \cite[Thm.~30.23]{villani2009} hold on $\mmms$ with uniform constants and, as in \autoref{Ex:QI}, transfer back qualitatively to $\mms$. A heat kernel control pair $(\Xi,\Sigma)$ for $(\mms,\Ch,\meas)$ is then given by
	\begin{align*}
		\Xi(x) &:= \meas[B_1(x)]^{-1},\\
		\Sigma(t) &:= t^{-N/2}.
	\end{align*}
\end{example}

\section{Taming for almost smooth LR manifolds}\label{Ch:Taming}

In this last chapter, we show that it is possible to introduce the notion of lower Ricci curvature bounds in the \emph{Kato class} $\Kato(\mms)$ (recall \autoref{Def:KatoIntro}) for suitable LR manifolds $(\mms,\Rmet)$, in the sense of \cite{ERST20}. To do so, we propose a smoothness condition, cf.~\autoref{Def:Almost smooth} below, on the given LR manifold that leads us to formulate the Bochner identity for this class of spaces.

The next \autoref{Sec:Taming} is devoted to briefly recall  the basic notions in the theory of \emph{tamed spaces} from \cite{ERST20} and to adapt them to our setting. We also review important properties of the Feynman--Kac semigroup associated to elements of $\Kato(\mms)$ and we prove the important criterion in \autoref{Th:Kato criterion}.

\subsection{Tamed spaces}\label{Sec:Taming} Various notions presented in this section hold for ``quasi-local distributions'' $\kappa$ in much greater generality. Since these are beyond our scope, we refer to \cite{ERST20} for further details.

\subsubsection{The Kato class on LR manifolds} For convenience, we restate \autoref{Def:KatoIntro}.

\begin{definition}\label{Def:KatoIntroII} The \emph{Kato class} $\Kato(\mms)$ consists of all signed Borel measures $\kappa$ on $\mms$ which do not charge $\Ch$-polar sets and such that
	\begin{align*}
	\lim_{t\to 0} \sup_{x\in\mms} \int_0^t\!\!\int_\mms\sfp(s,x,\cdot)\d\vert\kappa\vert\d s =0.
	\end{align*}
\end{definition}

Reminiscent of the smooth result \cite[Prop.~VI.10]{guneysu2017}, we now prove the following.

\begin{theorem}\label{Th:Kato criterion} Suppose that $\calk\colon \mms\to\R$ is a Borel function that can be decomposed as $\calk=\calk_1+\calk_2$  into the sum of two Borel functions $\calk_1,\calk_2\colon \mms\to \R$ such that
	\begin{enumerate}[label=\textnormal{\alph*.}]
		\item\label{La:a} $\calk_2\in\Ell^\infty(\mms)$, and
		\item\label{La:b} there exist $p\in [1,\infty)$ if $d=1$ or $p\in (d/2,\infty)$ if $d\geq 2$, and a heat kernel control pair $(\Xi,\Sigma)$ such that $\calk_1\in \Ell^p(\mms,\Xi\,\vol)$.
	\end{enumerate}
	Then for every $s>0$ and every $x\in\mms$,
	\begin{align}\label{Eq:Integrated heat kernel bd}
	\int_\mms\sfp(s,x,\cdot)\,\vert \calk\vert \d\vol\leq \Sigma^{1/p}(s)\,\Vert \calk_1\Vert_{\Ell^p(\mms,\Xi\,\vol)} + \Vert \calk_2\Vert_{\Ell^\infty(\mms)}.
	\end{align}
	In particular, for every choice of $(\Xi,\Sigma)$ and $p$ as in \ref{La:b} and every $\calk\in \Ell^p(\mms,\Xi\,\vol)+\Ell^\infty(\mms)$, we have $\calk\,\vol \in \Kato (\mms)$.
\end{theorem}

\begin{proof} We first derive the second claim from \eqref{Eq:Integrated heat kernel bd}. Decompose any $\calk\in \Ell^p(\mms,\Xi\,\vol)+\Ell^\infty(\mms,\vol)$ according to \ref{La:a} and \ref{La:b} above. Thus
	\begin{align*}
	&\lim_{t\to 0}\sup_{x\in\mms}\int_0^t\!\!\int_\mms\sfp(s,x,\cdot)\,\vert \calk\vert\d\vol\d s\\
	&\qquad\qquad \leq  \Vert \calk_1\Vert_{\Ell^p(\mms,\Xi\,\vol)}\,\lim_{t\to 0}\int_0^t \Sigma^{1/p}(s)\d s + \Vert \calk_2\Vert_{\Ell^\infty(\mms)}\, \lim_{t\to 0}t =0.
	\end{align*}
	
	We proceed to show the validity of \eqref{Eq:Integrated heat kernel bd}. Since $\vert \calk\vert \leq \vert \calk_1\vert + \vert \calk_2\vert$ on $\mms$ and since the heat kernel $\sfp(s,x,\cdot)$ is a sub-probability density for every $s>0$ and every $x\in\mms$, we may and will assume that $\calk_2$ vanishes identically on $\mms$. Then the inequality \eqref{Eq:Integrated heat kernel bd} for $p\in [1,\infty)$ if $d=1$ trivially follows from \eqref{Eq:Heat kernel bound control pair}, whence we concentrate on the situation $d\geq 2$, in which case we assume $p\in (d/2,\infty)$. Denote by $q := p/(p-1)$ the dual exponent to $p$, i.e.~$1/p+1/q=1$. Using Hölder's inequality and \eqref{Eq:Heat kernel bound control pair} yields
	\begin{align*}
	\int_\mms \sfp(s,x,\cdot)\,\vert \calk_1\vert\d\vol &= \int_\mms \sfp^{1/q}(s,x,\cdot)\,\sfp^{1/p}(s,x,\cdot)\,\vert \calk_1\vert\d\vol\\
	&\leq \Big[\!\int_\mms \sfp(s,x,\cdot)\d\vol\Big]^{1/q}\,\Big[\!\int_\mms \vert \calk_1\vert^p\,\sfp(s,\cdot,x)\d\vol\Big]^{1/p}\\
	&\leq \Big[\!\int_\mms \vert \calk_1\vert^p\,\Xi\,\Sigma(s)\d\meas\Big]^{1/p} = \Sigma^{1/p}(s)\,\Vert \calk_1\Vert_{\Ell^p(\mms,\Xi\,\vol)}.\qedhere
	\end{align*}
\end{proof}

\begin{remark} Since $\inf\Xi > 0$, we trivially have $\Ell^p(\mms,\Xi\,\vol) \subset\Ell^p(\mms)$.
\end{remark}

\begin{remark} Of course, \autoref{Th:Kato criterion} holds true on any (quasi-)regular, strongly local Dirichlet space $(\mms,\Ch,\meas)$ satisfying the absolute continuity hypothesis, i.e.~the existence of a heat kernel for $(\ChHeat_t)_{t\geq 0}$ \cite[Def.~A.2.16]{chen2012}.
\end{remark}

\begin{example} In the setting of \autoref{Ex:Spher symm} and \autoref{Ex:Spher symm II}, we have $\Ell^p(\mms_\psi) \subset\Kato(\mms_\psi)$ for every $p\in (d/2,\infty)$.
\end{example}

\subsubsection{Feynman--Kac semigroups induced by Kato class measures}
We denote by $\smash{((\Prob^x)_{x\in\mms}, (\sfb_t)_{t\in [0,\zeta)})}$ the $\vol$-reversible, continuous, strong Markov process associated with $(\Ch, W^{1,2}(\mms))$ and with explosion time $\zeta$. As detailed in \cite[Sec.~2.2, Sec.~2.4]{ERST20}, every $\kappa\in\Kato(\mms)$ induces an $\Ch$-quasi-local distribution on $W^{1,2}(\mms)$, and is  thus properly associated with a local continuous additive functional (AF, in short) $\smash{(\sfa_t^\kappa)_{t\in [0,\zeta)}}$, which is unique up to equivalence of AF's  \cite[Lem.~2.9]{ERST20}. For instance, if $\kappa = \calk\,\vol$ for  some nearly Borel function $\calk \in L^2(\mms)$, this AF can be explicitly represented as
\begin{align*}
\sfa^\kappa_t = \int_0^t f(\sfb_s) \, \d s.
\end{align*}

We associate to $\kappa \in \Kato(\mms)$ the \emph{Feynman--Kac semigroup} $(\sfP^\kappa_t)_{t \ge 0}$, defined in terms of the Feynman--Kac formula
\begin{equation}\label{eq:FKformula}
\sfP^\kappa_t f := \Exp^{\, \cdot} \big[\rme^{-\sfa_t^\kappa}\, f(\sfb_t)\,\One_{\{t < \zeta\}}\big]
\end{equation}
for every Borel function $f\colon\mms\to [0,\infty)$ and every $t \ge 0$. By Khasminskii's lemma \cite[Lem.~2.24]{ERST20}, the distribution induced by $\kappa$ is moderate \cite[Def.~2.13]{ERST20}, which implies that $(\Schr{\kappa}_t)_{t\geq 0}$ extends to a strongly continuous, exponentially bounded semigroup on $\Ell^p(\mms)$ for every $p\in [1,\infty]$ \cite[Rem.~2.14]{ERST20}.

\subsubsection{Singular Bakry--Émery condition}\label{Sub:Taming cond for dis} The key feature of $\Ch$-quasi-local distributions stemming from $\Kato(\mms)$ is that $\smash{(\Schr{q\kappa/2}_t)_{t\geq 0}}$, $q\in [1,\infty)$, can be treated with form techniques, which is briefly recapitulated here.

The following result from \cite[Cor.~2.25]{ERST20} is crucial for this purpose.

\begin{lemma}\label{Le:Form boundedness} For every $\rho > 0$, there exists $\alpha\in\R$ such that for every $f\in\F$,
	\begin{align*}
		\int_\mms \widetilde{f}^2\d\vert\kappa\vert \leq \rho\,\Ch(f) + \alpha\int_\mms f^2\d\vol.
	\end{align*}
\end{lemma}

Hence, by \cite[Thm.~2.49]{ERST20} the quadratic form
\begin{align*}
\Ch^{q\kappa/2}(f) := \Ch(f) + \frac{q}{2}\int_\mms\widetilde{f}^2\d\kappa
\end{align*}
is closed, lower semibounded in $L^2(\mms)$ and associated to $\smash{(\Schr{q\kappa/2}_t)_{t \ge 0}}$. Moreover, its domain $\smash{\calD(\Ch^{q\kappa/2}) := \{f \in \F : \int_\mms \widetilde{f}^2\d\vert\kappa\vert < \infty\}}$ corresponds to the whole $\F$. We call this form the \emph{taming energy}, and we denote by $\sfL^{q\kappa/2}$ the associated generator. We write $\Delta^{q\kappa} := 2\,\sfL^{q\kappa/2}$ and denote its domain by $\calD(\Delta^{q\kappa})$.

The following definition follows \cite[Def.~3.1]{ERST20}.

\begin{definition} We call $(\mms,\Rmet)$ \emph{tamed} \textnormal{(}through a given $\kappa\in\Kato(\mms)$\textnormal{)} if the Bakry--\smash{Émery} condition $\BE_1(\kappa,\infty)$ holds, that is, if for every $f\in \F$ with $\Delta f \in\F$ and every nonnegative $\phi\in \Dom(\Delta^{\kappa})$, we have
	\begin{align*}
		\int_\mms \Delta^\kappa\phi\,\vert\rmd f\vert^2\d\vol - \int\phi\,\Rmet^*(\rmd f,\rmd\Delta f)\,\vert\rmd f\vert^{-1}\d\vol \geq 0,
	\end{align*}
	where the second integral is taken over the set where $\vert \rmd f\vert > 0$.
\end{definition}




\subsection{An interior taming criterion}\label{Sub:An interior}

\begin{definition}\label{Def:Almost smooth} We call the LR manifold $(\mms,\Rmet)$ \emph{almost smooth} if there exists an open, $\vol$-conegligible subset $\mms^\infty$ of $\mms$ for which $\smash{(\mms^\infty, \Rmet\big\vert_{\mms^\infty})}$ is a smooth  Riemannian manifold.
\end{definition}

Reminiscent of \autoref{prop:Existence control pair}, \autoref{Cor:Quasi-isometry} and \autoref{Th:Kato criterion},  we now establish our main taming result, \autoref{Th:Tamed convex boundary}, for Kato Ricci bounds which are absolutely continuous w.r.t.~$\vol$.  The crucial step to prove \autoref{Th:Tamed convex boundary} is contained in  \autoref{Le:Reg lemma}. A similar strategy as in its proof has been pursued for different regularity results within a general second order calculus for tamed spaces, see e.g.~\cite[Cor.~5.12, Lem.~8.8]{braun2021}.

\begin{theorem}\label{Th:Tamed convex boundary} Let $(\Xi,\Sigma)$ be a given heat kernel control pair for the given almost smooth LR manifold $(\mms,\Rmet)$, and let $\calk\in\Ell^p(\mms,\Xi\,\vol)+\Ell^\infty(\mms)$ for some $p\in (d/2,\infty)$. Suppose that $\Delta$ is essentially self-adjoint on $\smash{\Cont_\comp^\infty(\mms^\infty)}$. Lastly, assume that the Ricci curvature of $\mms^\infty$ is bounded from below by $\calk$. Then $(\mms,\Rmet)$ is tamed by 
	\begin{align*}
		\kappa := \calk\,\vol.
	\end{align*}
\end{theorem}




\begin{lemma}\label{Le:Reg lemma} Under the assumptions of \autoref{Th:Tamed convex boundary}, if $f\in\Dom(\Delta)$ then $\vert\rmd f\vert\in W^{1,2}(\mms)$.
\end{lemma}

\begin{proof} 


The claim is straightforward if $f\in \Cont_\comp^\infty(\mms^\infty)$. Indeed, since the set $\mms\setminus \mms^\infty$ is $\vol$-ne\-gli\-gible, $\vert\rmd f\vert$ has a differential $\rmd\vert\rmd f\vert$ which vanishes $\vol$-a.e.~on $\{\vert\rmd f\vert = 0\}$, and
\begin{align}\label{Eq:Kato inequ}
		\big\vert\rmd \vert\rmd f\vert\big\vert \leq \big\vert\!\Hess f\big\vert_\HS\quad\vol\text{-a.e.}
\end{align}
by Kato's inequality for the Bochner Laplacian \cite[Prop.~2.2]{hess1980}, thus $\vert\rmd f\vert\in W^{1,2}(\mms)$.

To deduce the claim for general elements of $\Dom(\Delta)$ with Laplacian in $W^{1,2}(\mms)$, we first collect some preliminary considerations, still assuming that $\smash{f\in\Cont_\comp^\infty(\mms^\infty)}$.
Given any $\varepsilon > 0$, define $\eta_\varepsilon\in \Cont^\infty([0,\infty))$ through $\smash{\eta_\varepsilon(r) := (r+\varepsilon)^{1/2}}$. Note that $\smash{4\,\eta_\varepsilon''(r)\,r \geq -2\,\eta_\varepsilon'(r)}$ and $1/\eta_\varepsilon'(r) = 2\,\eta_\varepsilon(r)$ for every $r\geq 0$. Moreover, since $\vert \rmd f\vert^2\in\smash{\Cont_\comp^\infty(\mms^\infty)}$, successively applying the chain rule for the Laplacian and employing Bochner's identity on $\mms^\infty$ as well as \eqref{Eq:Kato inequ} yields
\begin{align}\label{Eq:Bochner 1}
\begin{split}
\Delta \big[\eta_\varepsilon\circ \vert\rmd f\vert^2\big] &= \big[\eta_\varepsilon'\circ\vert \rmd f\vert^2\big]\,\Delta\vert\rmd f\vert^2 + \big[\eta_\varepsilon''\circ\vert\rmd f\vert^2\big]\,\big\vert \rmd\vert\rmd f\vert^2\big\vert^2\\
&\geq 2\,\big[\eta_\varepsilon'\circ\vert \rmd f\vert^2\big]\,\big[\Rmet^*(\rmd f,\rmd\Delta f) + \calk\,\vert\rmd f\vert^2 + \big\vert\!\Hess f\big\vert_\HS^2\big]\\
&\qquad\qquad + 4\,\big[\eta_\varepsilon''\circ\vert\rmd f\vert^2\big]\,\vert\rmd f\vert^2\,\big\vert \rmd\vert\rmd f\vert\big\vert^2\\
&\geq 2\,\big[\eta_\varepsilon'\circ\vert \rmd f\vert^2\big]\,\big[\Rmet^*(\rmd f,\rmd\Delta f) + \calk\,\vert\rmd f\vert^2+ \big\vert\!\Hess f\big\vert_\HS^2 - \big\vert\rmd \vert\rmd f\vert\big\vert^2\big]\\
&\geq 2\,\big[\eta_\varepsilon'\circ\vert \rmd f\vert^2\big]\,\big[\Rmet^*(\rmd f,\rmd\Delta f) - \calk^-\,\vert\rmd f\vert^2\big]\quad\vol\text{-a.e.}\textcolor{white}{\big\vert_\HS^2}
\end{split}
\end{align}
Dividing by $2\,\eta_\varepsilon'\circ\vert\rmd f\vert^2$,  integrating the result on $\mms$ and using that $\Delta[\eta_\varepsilon\circ\vert\rmd f\vert^2]$ has compact support in $\smash{\mms^\infty}$ 
gives
\begin{align}\label{Eq:BOUND}
\int_\mms \big\vert\rmd\big[\eta_\varepsilon\circ \vert\rmd f\vert^2\big]\big\vert^2\d\vol & \leq -\int_\mms \Rmet^*(\rmd f,\rmd \Delta f) \d\vol + \int_\mms \calk^-\,\vert\rmd f\vert^2\d\vol.
\end{align}
By \autoref{Th:Kato criterion}, the measure $\calk^-\,\vol$ belongs to $\Kato(\mms)$. Since $\vert\rmd f\vert\in W^{1,2}(\mms)$, by \autoref{Le:Form boundedness} there exists $\alpha\in \R$ (independent of $f$) such that
\begin{align*}
\int_\mms \calk^-\,\vert\rmd f\vert^2\d\vol \leq \frac{1}{2}\int_\mms \big\vert\rmd\vert\rmd f\vert^2\big\vert^2\d\vol + \alpha \int_\mms \vert\rmd f\vert^2 \d\vol.
\end{align*}
Letting $\varepsilon\to 0$ in \eqref{Eq:BOUND} and employing the lower semicontinuity of  $\Ch$ w.r.t.~pointwise $\vol$-a.e.~convergence yields
\begin{align}\label{Eq:IBoch}
\int_\mms \big\vert\rmd\vert\rmd f\vert\big\vert^2\d\vol \leq 2\int_\mms (\Delta f)^2\d\vol + 2\alpha\int_\mms \vert\rmd f\vert^2\d\vol.
\end{align}

Now we  prove the claim for arbitrary $f\in\Dom(\Delta)$. By assumption, there exists a sequence $(f_n)_{n\in\N}$ in $\smash{\Cont_\comp^\infty(\mms^\infty)}$ converging to $f$ in the $\Ell^2$-graph norm of $\Delta$. We may and will assume without restriction that $\vert\rmd f_n\vert\to\vert\rmd f\vert$ pointwise $\vol$-a.e.~as $n\to\infty$. Then, since $(\vert\rmd f_n\vert)_{n\in\N}$ is a bounded sequence in $W^{1,2}(\mms)$ by \eqref{Eq:IBoch}, the claim follows again by lower semicontinuity of $\Ch$.
\end{proof}

\begin{proof}[Proof of \autoref{Th:Tamed convex boundary}]
As seen in the proof of \autoref{Le:Reg lemma}, by \autoref{Cor:Quasi-isometry} it follows that $\calk\,\vol\in \Kato(\mms)$. 
To derive $\BE_1(\kappa,\infty)$, let $f\in \smash{\bigcup_{t>0} \ChHeat_t\Ell^2(\mms)}$. By local elliptic regularity theory, see e.g.~\cite{grigoryan2009}, we have $\smash{f\in \Cont^\infty(\mms^\infty)}$, whence
\begin{align*}
	\Delta\big[\eta_\varepsilon\circ \vert \rmd f\vert^2\big] \geq 2\,\big[\eta_\varepsilon'\circ \vert\rmd f\vert^2\big]\,\big[\Rmet^*(\rmd f,\rmd\Delta f) + \calk\,\vert\rmd f\vert^2\big]\quad\vol\text{-a.e.}
\end{align*}
for every $\varepsilon >0$ by a similar argument as for \eqref{Eq:Bochner 1}, retaining the notations used above. Multiplying this inequality by any given nonnegative, boundedly supported $\phi \in \Lip(\mms)$ and integrating by parts while using \autoref{Le:Reg lemma} yields
\begin{align*}
	0 &\leq -\int_\mms \Rmet^*(\rmd\phi,\rmd\vert \rmd f\vert)\,\big[\eta_\varepsilon'\circ\vert\rmd f\vert^2\big]\,\vert\rmd f\vert\d\vol - \int_\mms \calk\,\phi\,\big[\eta_{\varepsilon}'\circ\vert\rmd f\vert^2\big]\,\vert\rmd f\vert^2\d\vol\\
	&\qquad\qquad - \int_\mms\phi\,\Rmet^*(\rmd f,\rmd \Delta f)\,\big[\eta_{\varepsilon}'\circ\vert\rmd f\vert^2\big]\d\vol\\
	&\leq -\int_\mms \Rmet^*(\rmd\phi,\rmd\vert \rmd f\vert)\,\big[\eta_\varepsilon'\circ\vert\rmd f\vert^2\big]\,\vert\rmd f\vert\d\vol\\
	&\qquad\qquad  - \int_{\mms} \min\{\calk^+, R\}\,\phi\,\big[\eta_{\varepsilon}'\circ\vert\rmd f\vert^2\big]\,\vert\rmd f\vert^2\d\vol  + \frac{1}{2}\int_{\mms}\calk^- \phi\,\vert\rmd f\vert\d\vol\\
	&\qquad\qquad - \int_\mms\phi\,\Rmet^*(\rmd f,\rmd \Delta f)\,\big[\eta_{\varepsilon}'\circ\vert\rmd f\vert^2\big]\d\vol
	\end{align*}
	for any fixed $R>0$. Setting $\smash{\kappa_R := \min\{\calk^+,R\}\,\meas - \calk^-\,\meas\in\Kato(\mms)}$, as $\varepsilon \to 0$, by Lebesgue's theorem this reduces to
\begin{align*}
0 &\leq -\Ch^{\kappa_R}(\phi,\vert\rmd f\vert) - \int \phi\,\Rmet^*(\rmd f,\rmd \Delta f)\,\vert\rmd f\vert^{-1}\d\vol
\end{align*}
and hence, as $R\to \infty$ by Levi's theorem, to
\begin{align*}
0 &\leq -\Ch^\kappa(\phi,\vert\rmd f\vert) - \int \phi\,\Rmet^*(\rmd f,\rmd \Delta f)\,\vert\rmd f\vert^{-1}\d\vol.
\end{align*}
In both cases, the last integrals are taken over the set where $\vert\rmd f\vert > 0$. By density of boundedly supported Lipschitz functions in $W^{1,2}(\mms)$ outlined in \autoref{Sub:Weak}, this inequality readily extends to all nonnegative $\phi\in W^{1,2}(\mms)$. Indeed, the map $g \mapsto \smash{\int_{\mms} \calk\,g\,\vert\rmd f\vert\d\vol}$ is continuous in $W^{1,2}(\mms)$ by Cauchy--Schwarz's inequality, \autoref{Le:Reg lemma} and \autoref{Le:Form boundedness}.

If in particular $\phi\in \Dom(\Delta^\kappa)$, which is a subset of $W^{1,2}(\mms)$ by \autoref{Sub:Taming cond for dis}, integration by parts again gives
\begin{align*}
	\int_\mms \Delta^\kappa \phi\,\vert\rmd f\vert\d\vol - \int \phi\,\Rmet^*(\rmd f,\rmd\Delta f)\,\vert\rmd f\vert^{-1}\d\vol \geq 0.
\end{align*}
To deduce $\BE_1(\kappa,\infty)$ it remains to follow the proof of \cite[Thm.~3.4]{ERST20}.
\end{proof}

We quickly discuss two situations in which essential self-adjointness of $\Delta$ holds. To couple this with lower Ricci bounds, one can e.g.~take the space $(\mms,\Rmet)$ under consideration to be (quasi-isometric to) a smooth, geodesically complete Riemannian manifold with uniformly lower bounded Ricci curvature, according to the discussions from \autoref{Ex:QI} and \autoref{Ex:QII}. 

\begin{example}[Removal of codimension 4 sets]\label{Ex:Sing I} Endow $\mms := \calM\setminus S$, where $\calM$ is a smooth manifold of dimension $d\geq 4$ and $S$ is a closed submanifold of codimension at least $4$, with a complete Riemannian metric $\Rmet$.  Then the induced Laplacian $\Delta$ is essentially self-adjoint on $\smash{\Cont_\comp^\infty(\mms)}$ \cite[Thm.~3]{masamune1999}.
	
	In the same spirit, one can take $S$ to be the finite union of compact submanifolds, not necessarily all of the same dimension, with codimension at least $4$ \cite[Thm.~5]{prandi2018}. The finiteness and the compactness condition can be weakened \cite[Rem.~3]{prandi2018}, which allows e.g.~for perforated domains such as $\smash{\R^5\setminus \Z^5}$ (which obeys $\BE_1(0,\infty)$).
	
	
\end{example}

\begin{example}[Quantum confinements and ARS]\label{Ex:Sing II} Let $(\mms,\Rmet)$ be a smooth, noncomplete Riemannian manifold, and denote by $\smash{(\overline{\mms}, \overline{\met})}$ its metric completion. Suppose that the distance function $\delta\colon \mms \to [0,\infty)$ given by
	\begin{align*}
		\delta(x) := \inf\{ \overline{\met}(x,y) : y\in \overline{\mms}\setminus \mms \}
	\end{align*}
	has regularity $\smash{\Cont^2}$ on $\delta^{-1}((0,\varepsilon])$ for some $\varepsilon > 0$. Lastly, assume that there exists some  $\varrho>0$ such that
	\begin{align*}
		(\Delta\delta)^2 - 2\,\big\vert\!\Hess f\big\vert_{\HS}^2 - 2\,\Ric(\nabla \delta,\nabla \delta)\geq \frac{3}{\delta^2} - \frac{4\varrho}{\delta}\quad\text{on }\delta^{-1}((0,\varepsilon]).
	\end{align*}
	Then $\Delta$ is essentially self-adjoint on $\smash{\Cont_\comp^\infty(\mms)}$ \cite[Thm.~1]{prandi2018}.
	
	Examples of spaces satisfying these conditions are cones, metric horns, or anti-cones \cite{prandi2018}. In fact \cite[Thm.~8]{prandi2018}, the above hypotheses cover  \emph{regular almost Riemannian structures} \cite[Def.~7.1, Def.~7.10]{prandi2018}, briefly regular ARS. These structures, roughly said,  consist of a smooth manifold $\mms$ endowed with a metric $\Rmet$ which is singular on an embedded hypersurface, and smooth on its complement.
\end{example}


Let us finally list some corollaries of \autoref{Th:Tamed convex boundary}. 
\autoref{Cor:GE} follows from the proof of \autoref{Th:Tamed convex boundary} and \cite[Thm.~3.4]{ERST20}. A short proof of \autoref{Cor:SC} is included for convenience. The integrated Bochner inequality from \autoref{Cor:Hess} --- according to the notion of Hessian established in \cite{braun2021} --- follows from \cite[Cor.~8.3]{braun2021}.

\begin{corollary}\label{Cor:GE} In the situation of \autoref{Th:Tamed convex boundary}, the heat flow satisfies the following gradient estimate for every $f\in W^{1,2}(\mms)$ and every $t\geq 0$:
\begin{align*}
\vert\rmd\ChHeat_t f\vert \leq \Schr{\kappa/2}_t\vert\rmd f\vert\quad\vol\text{-a.e.}
\end{align*}	
\end{corollary}

\begin{corollary}\label{Cor:SC} If in addition to the hypotheses in \autoref{Th:Tamed convex boundary}, $(\mms,\met)$ is complete, then its heat flow is conservative.
\end{corollary}

\begin{proof} Since the intrinsic distance generated by $\Ch$ generates the topology of $\mms$ by \autoref{Sub:Heat kernel} and since $(\mms,\met)$ is complete, the Dirichlet space induced by $(\mms,\Rmet)$ is intrinsically complete according to \cite[Def.~3.8]{ERST20} as proven in \cite[Rem.~3.8]{ERST20}. Since every tamed space is weakly tamed \cite[Def.~3.10]{ERST20}, conservativeness then follows from \cite[Thm.~3.11]{ERST20}.
\end{proof}

\begin{corollary}\label{Cor:Hess} In the situation of \autoref{Th:Tamed convex boundary}, every function $f\in\Dom(\Delta)$ belongs to  $\Dom_{\textnormal{reg}}(\Hess)$, and we have
	\begin{align*}
		\int_\mms \big\vert\!\Hess f\big\vert_\HS^2\d\vol \leq \int_\mms (\Delta f)^2\d\vol -\int_\mms \calk\,\vert\rmd f\vert^2\d\vol.
	\end{align*}
\end{corollary}

\subsection{A boundary taming criterion}\label{Sub:A boundary} Finally,  in \autoref{Th:Surface measure} we establish a condition for weighted surface measures on the boundary of  special LR manifolds to belong to $\Kato(\mms)$. With this at hand, in \autoref{Th:Taming bdry qi} we provide a taming condition for LR manifolds with boundary.

Throughout this section, our setting is the following: suppose that $(\mms,\Rmet)$ is quasi-isometric to a smooth,  complete Riemannian manifold $(\mmms,\RRmet)$ of dimension $d\geq 3$, without or with convex boundary, and with uniformly lower bounded Ricci curvature by a bi-Lipschitz map $F\colon \mms\to \mmms$. Moreover, let us assume that $\mmms$ has positive injectivity radius. As detailed in \autoref{Re:EINS} and \autoref{Re:ZWEI}, this assumption will allow us to characterize the Kato class of $\mmms$ in metric terms à la Aizenman--Simon \cite{aizenman1982,kuwae}, a formulation which is far easier to control in terms of codimension one sets. Using lower bounds for $\sfp_\RRmet$ \cite{sturm1992}, this information is then ``pulled back'' to $\mms$ through $F$.

\subsubsection{Correspondence of the Kato classes}\label{Sub:Corresp}

\begin{proposition}\label{Pr:Kato identification} For every $\mu\in\Kato(\mmms)$, we have $\smash{(F^{-1})_\push\mu\in \Kato(\mms)}$.
\end{proposition}

\begin{proof} By bi-Lipschitz continuity of $F$ and \autoref{Le:Quasi-iso Dirichlet}, the signed Borel measure $\kappa := \smash{(F^{-1})_\push\mu}$ on $\mms$ does not charge $\Ch_\Rmet$-polar sets. Now note that
\begin{align*}
&\sup_{x\in\mms}\int_0^t\!\!\int_{\mms} \sfp_\Rmet(s,x,\cdot)\d\vert\kappa\vert\d s\\
&\qquad\qquad \leq C \sup_{x\in\mms}\int_0^t\!\!\int_{\mms}\,\vol_\Rmet\big[\Ball_{\!\sqrt{t}}^\Rmet(x)\big]^{-1}\,\exp\!\Big[\!-\!\frac{\met_\Rmet^2(x,\cdot)}{5t}\Big]\d\vert\kappa\vert\d s\\
&\qquad\qquad \leq C\sup_{z\in\mmms}\int_0^t\!\!\int_{\mmms} \vol_\RRmet\big[\Ball_{\!\sqrt{t}/C}^\RRmet(z)\big]^{-1}\,\exp\!\Big[\!-\!\frac{\met_\RRmet^2(z,\cdot)}{5Ct}\Big]\d\vert\mu\vert\d s\\
&\qquad\qquad \leq C\sup_{z\in\mmms}\int_0^{Ct}\!\!\int_{\mmms} \sfp_\RRmet(s,z,\cdot)\d\vert\mu\vert \d s.
\end{align*}
Here we have used \eqref{Eq:HK bound II} in the first inequality, \autoref{Re:Volume comparison} and the Lipschitz continuity of $F$ in the second one, and the local doubling property of $\vol_\RRmet$ \cite[Thm.~5.6.4]{saloffcoste2002} and Gaussian lower bounds on $\sfp_\RRmet$ \cite[Thm.~4.5]{sturm1992} in the last step. Moreover, we allow the constant $C>0$, which does not depend on $x$, $z$ or $t$, to change from line to line. These estimates readily yield the claim.
\end{proof}

\begin{remark}\label{Re:EINS} \autoref{Pr:Kato identification} is a variant of \cite[Lem.~2.33]{ERST20}. However, here $(\mms,\Rmet)$ is not assumed to be of Harnack-type, hence to admit Gaussian upper and lower bounds on $\sfp$. Instead, we use this property on a concrete bi-Lipschitz image of $\mms$.
\end{remark}

Now let $Y\subset\mms$ be an open subset having  \emph{weakly Lipschitz} boundary, i.e.~there exist some $c>0$, a  covering $U_1,\dots,U_k$, $k\in\N$, of $\partial Y$ by open subsets of $\mms$, and $c$-Lip\-schitz maps $\smash{\varphi_i \colon U_i\to\R^{d-1}}$ such that
\begin{align*}
(\varphi_i)_\push\sigma_Y \leq c\,\Leb^{d-1} \quad\text{on }\varphi_i(U_i\cap \partial Y)
\end{align*}
for every $i\in\{1,\dots,k\}$. Here $\sigma_Y$ is the surface measure of $\partial Y$.

\begin{example} A simple framework in which the above hypotheses hold is when $Y=\mms$, under the assumption of compactness of $\partial N$ (hence $\partial\mms$).
	
	Moreover, every Lipschitz domain in $\smash{\R^d}$ with, say, compact boundary has weakly Lipschitz boundary.
\end{example}

\begin{theorem}\label{Th:Surface measure} Retain the previous  assumptions and notations, and suppose that $\call\in\Ell^q(\partial Y,\sigma_Y)$ for some $q \in (d-1,\infty)$. Then $\kappa := \call\,\sigma_Y$ belongs to $\Kato(\mms)$.
\end{theorem}

\begin{proof} We first claim that
\begin{align}\label{Eq:Claim QI}
\lim_{r\to 0} \sup_{x\in\mms}\int_{\Ball_r^\Rmet(x)} \met_\Rmet^{2-d}(x,\cdot)\,\vert\call\vert\d\sigma_Y = 0.
\end{align}
Indeed, given any $x\in\mms$, for $p\in (1, \infty)$ satisfying $1/p+1/q=1$ we have
\begin{align*}
\int_{\Ball_r^\Rmet(x)}\met_\Rmet^{2-d}(x,\cdot)\,\vert\call\vert\d\sigma_Y  \leq \Vert\call\Vert_{\Ell^p(\partial Y,\sigma)}\,\Big[\!\int_{\Ball_r^\Rmet(x)\cap U} \met_\Rmet^{q(2-d)}(x,\cdot)\d\sigma_Y\Big]^{1/q},
\end{align*}
where $U := U_1\cup\dots\cup U_k$. For every $i\in\{1,\dots,k\}$, we have
\begin{align*}
\int_{\Ball_r^\Rmet(x)\cap U_i} \met_\Rmet^{q(2-d)}(x,\cdot)\d\sigma_Y &\leq  c^{q(d-2)} \int_{V_i}\big\vert\varphi_i -\varphi_i(x)\big\vert^{q(2-d)}\d\sigma_Y\\
&\leq c^{q(d-2)+1} \int_{B_{cr}(0)} \vert \cdot\vert^{q(2-d)}\d\Leb^{d-1}.
\end{align*}
Here we write $V_i := \vert\varphi_i-\varphi_i(x)\vert^{-1}([0,cr))\cap U_i$, and $B_{cr}(0)$ is intended as open ball in $\smash{\R}^{d-1}$. As $r \to 0$, the last integral tends to zero precisely by our choice of $q$. This readily proves the claim \eqref{Eq:Claim QI}.

Now we finally argue that \eqref{Eq:Claim QI} implies $\call\,\sigma_Y\in\Kato(\mms)$. Indeed, thanks to \eqref{Eq:Claim QI}, \autoref{Le:Rev char} and \autoref{Re:Volume comparison},
\begin{align*}
\lim_{r\to 0}\sup_{z\in\mmms}\int_{\Ball_r^\RRmet(z)} \met_\RRmet^{2-d}(z,\cdot)\,\vert\call\vert\circ F^{-1}\d F_\push\sigma_Y =0.
\end{align*}
By our assumptions on $(\mmms,\RRmet)$, the corresponding  considerations in \cite[Ex.~6.8]{kuwae} and \cite[Thm.~3.1]{kuwae}, the previous property implies --- in fact, is equivalent to --- the statement that $[\vert\call\vert\circ F^{-1}]\,F_\push\sigma_Y\in \Kato(\mmms)$ according to \autoref{Def:KatoIntroII}. The assertion follows from \autoref{Pr:Kato identification}.
\end{proof}

\begin{remark}\label{Re:ZWEI} \autoref{Th:IntroTh1} is a variant of  \cite[Thm.~2.36]{ERST20}. Therein, however, $\mms$ is supposed to be smooth, and the equivalence of \eqref{Eq:Claim QI} with the condition that  $\call\,\sigma_Y\in\Kato(\mms)$ is part of the \emph{assumption} as well. In our setting, we do not know if this equivalence holds (unless $\sfp_\Rmet$ satisfies Gaussian lower bounds). Again, instead, we use this correspondence on a concrete bi-Lipschitz image of $\mms$.
\end{remark}

\subsubsection{Taming of inner uniform domains} Finally, we study taming of sets $Y$ as above, endowed with the canonical Dirichlet structure $(\Ch_Y, W^{1,2}(Y))$ inherited by the LR structure of $\smash{(Y,\Rmet\big\vert_Y)}$, as outlined in \autoref{Sub:Dir form}. In particular, all objects from \autoref{Sub:Calculus on} associated to $\smash{(Y,\Rmet\big\vert_Y)}$ will be assigned (in a notationally evident way) with a sub- or superscript $Y$.

Partly following the terminology proposed in \cite{GSL11}, in addition to the assumptions on $Y$ from \autoref{Sub:Corresp}, we make the following hypotheses.
\begin{enumerate}[label=\alph*.]
	\item\label{La:EINS} \emph{Inner uniformity.} There exist some constants $c, C > 0$ such that any two points $x, y \in \mms$ can be connected by a continuous curve $\gamma\colon [0, 1]\to Y$ with length at most $C\, \met_Y(x, y)$ with the property that for every $z \in \gamma([0, 1])$, 
	\begin{align*}
	\met(z, \partial \mms) \ge c \min\{\met_Y (z, x), \met_Y(z, y) \}.
	\end{align*}
	\item\label{La:ZWEI} \emph{Tubular volume.} We have
	\begin{align*}
		\inf\!\Big\{\dfrac{\vol[\Ball_r(y) \cap Y]}{\vol[\Ball_r(y)]} : r > 0,\, y \in Y \Big\} > 0
	\end{align*}
	\item\label{La:DREI} \emph{Regular exhaustion.} There exists an increasing sequence $(Y_n)_{n\in\N}$ of subsets of $Y$ with the following properties.
	\begin{itemize}
		\item Every $Y_n$, $n\in\N$, has a smooth structure and 
		smooth boundary such that $\smash{\Rmet\big\vert_{Y_n}}$ is smooth.
		\item The sequence $\smash{(\overline{Y}_n)_{n\in\N}}$ is an $\Ch_Y$-nest consisting of compact sets.
		\item For every compact $K\subset Y$ there exists
		$N\in\N$ such that $K\subset Y_n$ for every $n\geq N$.
	\end{itemize}
	\item\label{La:VIER} \emph{Curvature bounds.} There exist  $\calk\in\Ell^p(Y,\vol_Y[\Ball_1^Y(\cdot)]^{-1}\,\vol_Y)+ \Ell^\infty(Y,\vol_Y)$, $p\in (d/2,\infty)$,  and $\call\in\Ell^q(\partial Y,\sigma_Y)$, $q \in (d-1,\infty)$, as well as a sequence $(\call_n)_{n\in\N}$ of functions $\call_n \colon \partial Y_n\to \R$, $n\in\N$, such that for every $n\in\N$,
	\begin{alignat*}{3}
		\Ric_{Y_n} &\geq \calk &&&\quad &\text{on }Y_n,\\
		\call_n &= \call &&&& \text{on }\partial Y_n\cap\partial Y,\\
		\call_n &\geq 0 &&&& \text{on }\partial Y_n\setminus\partial Y.
	\end{alignat*}
\end{enumerate} 

\begin{theorem}\label{Th:Taming bdry qi} Under the foregoing hypotheses, $(Y,\smash{\Rmet}\big\vert_Y)$ is tamed by
	\begin{align*}
		\kappa := \calk\,\vol_Y + \call\,\sigma_Y.
	\end{align*}
\end{theorem}

\begin{proof} First, we claim that $\calk\,\vol_Y \in \Kato(Y)$, which will imply that $\kappa\in\Kato(Y)$ thanks to \autoref{Th:Surface measure}. Indeed, by \ref{La:EINS} and \ref{La:ZWEI} (which grant that distances and volumes of balls in $Y$ are comparable with those in $\mms$, cf.~also \cite[Lem.~2.34]{ERST20}) as well as the considerations from \autoref{Ex:QI}, the assumptions of \autoref{Th:Quasi-isometry} are satisfied for $\smash{(Y,\Rmet\big\vert_Y)}$, and the claim follows from \autoref{Th:Kato criterion}.
	
We set $\call^*_n := \call_n \,\One_{\partial Y_n \cap \partial Y}$,  where $n\in\N$. Thanks to \ref{La:VIER} and \autoref{Th:Surface measure}, we obtain  $\smash{\call^*_n\,\sigma_{Y_n} \in \Kato(Y_n)}$. Furthermore, the AF associated with $\smash{\calk^-\,\vol_{Y_n}+ (\call_n^*)^-\,\sigma_{Y_n}}$ is no larger than the AF corresponding to $\smash{\calk^-\,\vol_Y + \call^-\,\sigma_Y}$ up to the first hitting time of $\partial Y_n\setminus \partial Y$. By Khasminskii's lemma \cite[Lem.~2.24]{ERST20}, this implies that the family $(\kappa_n)_{n\in\N}$ given by $\smash{\kappa_n := \calk\,\vol_{Y_n} + \call_n\,\sigma_{Y_n}\in\Kato(Y_n)}$ is uniformly $1$- and $2$-moderate according to \cite[Thm.~4.5]{ERST20}. Together with \ref{La:DREI}, the same result ensures the taming condition for $\smash{(Y,\Rmet\big\vert_Y)}$ and hence the claim.
\end{proof}

	\end{document}